\documentclass[a4paper]{amsart}
\usepackage{color}
\usepackage[latin1]{inputenc}
\usepackage[T1]{fontenc}
\usepackage{amsfonts}
\usepackage{amssymb}
\usepackage{amsmath}
\usepackage{amsthm}
\usepackage{stmaryrd}
\usepackage{enumerate}
\usepackage{multirow}
\usepackage{graphicx}
\usepackage{setspace}
\usepackage{graphicx,type1cm,eso-pic,color}
\usepackage{pstricks}
\usepackage{float}
\newtheorem{theorem}{Theorem}[section]

\newtheorem{proposition}[theorem]{Proposition}
\newtheorem{corollary}[theorem]{Corollary}
\newtheorem{definition}[theorem]{Definition}

\newtheorem{remark}[theorem]{Remark}
\newtheorem{example}[theorem]{Example}
\begin{document}
\setlength\arraycolsep{2pt}
\title{Nonparametric Estimation from Correlated Copies of a Drifted Process}
\author{Nicolas MARIE$^{\dag}$}
\email{nmarie@parisnanterre.fr}
\date{}
\maketitle
\noindent
$^{\dag}$Universit\'e Paris Nanterre, CNRS, Modal'X, 92001 Nanterre, France.
\maketitle
%


%
\begin{abstract}
This paper presents several situations leading to the observation of multiple correlated copies of a drifted process, and then non-asymptotic risk bounds are established on nonparametric estimators of the drift function $b_0$ and its derivative. For drifted Gaussian processes with a regular enough covariance function, a sharper risk bound is established on the estimator of $b_0'$, and a model selection procedure is provided with theoretical guarantees.
\end{abstract}
\tableofcontents
%


%
\section{Introduction}\label{section_introduction}
Let $X = (X_t)_{t\in\mathbb R_+}$ be the process defined by
\begin{equation}\label{main_model}
X_t = b_0(t) + Z_t
\textrm{ $;$ }\forall t\in\mathbb R_+,
\end{equation}
where $b_0 :\mathbb R_+\rightarrow\mathbb R$ is an unknown continuous function, and $Z = (Z_t)_{t\in\mathbb R_+}$ is a centered continuous process such that $Z_0 = 0$ and $\mathbb E(Z_{t}^{2}) <\infty$ for every $t\in [0,T]$. Consider also $T > 0$, $N\in\mathbb N^* =\mathbb N\backslash\{0\}$ and the processes $X^1,\dots,X^N$ defined by
\begin{equation}\label{main_model_copies}
X_{t}^{i} = b_0(t) + Z_{t}^{i}
\textrm{ $;$ }
(i,t)\in\{1,\dots,N\}\times [0,T],
\end{equation}
where $Z^1,\dots,Z^N$ are $N$ copies of $Z$ such that, for every $i,k\in\{1,\dots,N\}$,
\begin{equation}\label{noise_covariance_condition}
\left|\int_{0}^{T}\mathbb E(Z_{s}^{i}Z_{s}^{k})ds\right|\leqslant \Gamma^{i,k}
\quad {\rm with}\quad
\Gamma^{i,k}\geqslant 0,
\quad {\rm and}\quad
\sum_{i,k = 1}^{N}\Gamma^{i,k}\underset{N\rightarrow\infty}{=}\textrm o(N^2).
\end{equation}
Note, and this is crucial, that Model (\ref{main_model_copies}) and the condition (\ref{noise_covariance_condition}) allow to consider dependent copies of the process $X$. Several dynamics and observation schemes lead to Model (\ref{main_model_copies}) for well-chosen $b_0$, $Z$ and $\Gamma = (\Gamma^{i,k})_{i,k}$. First, $N$ correlated copies $Y^1,\dots,Y^N$ of the solution $Y$ of a non-autonomous linear stochastic differential equation driven by a (fractional) Brownian motion - called a linear (fractional) diffusion in the sequel for the sake of simplicity - can be written as $Y^i = Y_0\exp(X^i)$, where the corresponding $b_0$, $Z$ and $\Gamma$ are provided in Section \ref{subsubsection_copies_linear_diffusion}. For instance, $Y^1,\dots,Y^N$ are appropriate to model the prices of $N$ interacting risky assets of the same kind as in Duellmann et al. \cite{DKK10}, or the elimination processes of a drug by $N$ patients involved in a clinical trial (see Donnet and Samson \cite{DS13}). In the same spirit, Section \ref{subsubsection_interacting_particle_system} provides a basic interacting particle system leading to Model (\ref{main_model_copies}). Now, in statistical inference for stochastic differential equations, at least two kinds of estimators of the drift function are investigated in the literature: those computed from a single long-time observation of the ergodic stationary solution (see Kutoyants \cite{KUTOYANTS04} and Kubilius et al. \cite{KMR17}), and those computed from multiple copies of the solution observed on a short-time interval (see Comte and Genon-Catalot \cite{CGC20}, Denis et al. \cite{DDM21}, Marie \cite{MARIE25}, etc.). In our simple Model (\ref{main_model}), when $Z$ the fractional Brownian motion, Section \ref{subsection_examples_observation_scheme} shows how to construct - with theoretical guarantees - correlated copies of $X_{|[0,T]}$ from one long-time observation of $X$. So, Section \ref{subsection_examples_observation_scheme} shows that a single long-time observation of $X$ leads to our Model (\ref{main_model_copies}) when $Z$ is a fractional Brownian motion, but note that the situations presented in Sections \ref{subsubsection_copies_linear_diffusion} and \ref{subsubsection_interacting_particle_system} - related to functional data-analysis - don't lead to a single long-time observation of $X$.
\\
\\
Beyond the aforementioned possible applications of Model (\ref{main_model_copies}), the estimation of $b_0$ and $b_0'$ (when $b_0\in C^1(\mathbb R_+;\mathbb R)$) - which is part of nonparametric functional data-analysis (see Ferraty and Vieu \cite{FV06}) - need to be investigated. On nonparametric estimators of $b_0$ computed from one long-time observation of $X$ when $Z$ is a Gaussian process, see Ibragimov and Rozanov \cite{IR78}, Chapter VII, or the more recent paper \cite{PR08} written by Privault and R\'eveillac, and on those computed from multiple independent copies of $X$ observed on a short-time - possibly random - interval, the reader can refer to Bunea et al. \cite{BIW11}, Kassi and Patilea \cite{KP25} and references therein. On kernel-based nonparametric estimators of $b_0'$ computed from multiple independent copies of $X$, see Ghale-Joogh and Hosseini-Nasab \cite{GJHN21} and references therein. On asymptotic results on projection estimators - in the B-spline basis - of all derivatives of $b_0$ in Model (\ref{main_model_copies}), the reader can refer to Cao \cite{CAO14}. Up to our knowledge, \cite{CAO14} is the only reference on a nonparametric estimator of $b_0'$ in Model (\ref{main_model_copies}) driven by correlated signals.\\
In our paper, a simple estimator of $b_0$ (on $[0,T]$) is given by $\widehat b =\overline X :=\frac{1}{N}\sum_{i = 1}^{N}X^i$ and reaches the parametric rate $\mathfrak R_N =\frac{1}{N^2}\sum_{i,k = 1}^{N}\Gamma^{i,k}$ (see Section \ref{subsection_estimator_b_0}). However, having in mind the estimation of the drift function of a non-autonomous linear (fractional) diffusion process, a suitable estimator of $b_0'$ (on $[0,T]$) is also required. So, now, assume that $b_0\in C^1(\mathbb R_+;\mathbb R)$, and that the paths of $Z$ are locally $\alpha$-H\"older continuous with $\alpha\in (0,1)$. Moreover, consider $\mathcal S_m = {\rm span}\{\varphi_1,\dots,\varphi_m\}$, where $m\in\{1,\dots,N\}$, $\varphi_1,\dots,\varphi_N$ are continuously differentiable functions from $[0,T]$ into $\mathbb R$, and $(\varphi_1,\dots,\varphi_N)$ is an orthonormal family of $\mathbb L^2([0,T],dt)$. In the sequel, the orthogonal projection from $\mathbb L^2([0,T],dt)$ onto $\mathcal S_m$ is denoted by $\Pi_m(\cdot)$. Since $\widehat b$ is an estimator of $b_0$ reaching the parametric rate $\mathfrak R_N$, so is $\widehat b_m =\Pi_m(\overline X)$ for $m$ large enough, and since
\begin{displaymath}
\widehat b_m =
\sum_{j = 1}^{m}\left(\int_{0}^{T}\varphi_j(s)d\mathbb X_s\right)\varphi_j
\quad {\rm with}\quad
\mathbb X =\int_{0}^{.}\overline X_sds,
\end{displaymath}
a natural estimator of $b_0'$ is given by
\begin{displaymath}
\widehat b_m' =\sum_{j = 1}^{m}
\left(\int_{0}^{T}\varphi_j(s)d\overline X_s\right)\varphi_j,
\end{displaymath}
where the integral with respect to $\overline X$ is taken in the sense of Young. Note that the paths of $\overline X$ are $\alpha$-H\"older continuous from $[0,T]$ into $\mathbb R$ because $b_0\in C^1(\mathbb R_+;\mathbb R)$ and the paths of the $Z^i$'s are locally $\alpha$-H\"older continuous, and since the $\varphi_j$'s are continuously differentiable from $[0,T]$ into $\mathbb R$, the Young integral $\int_{[0,T]}\varphi_jd\overline X(\omega)$ exists for every $j\in\{1,\dots,m\}$ and $\omega\in\Omega$ (see Section \ref{section_Young_integral}). In Section \ref{subsection_estimator_b_0'}, a non-asymptotic risk bound is established on $\widehat b_m'$, and then improved when $Z$ is a Gaussian process with a regular enough covariance function. A model selection procedure is provided with theoretical guarantees on the corresponding adaptive estimator. In Section \ref{section_estimation_examples}, the results established on $\widehat b$ and $\widehat b_m'$ in Sections \ref{subsection_estimator_b_0} and \ref{subsection_estimator_b_0'} are applied in the models of Section \ref{subsection_examples_dynamics}.
\\
\\
The outline of the paper is as follows: Section \ref{section_Young_integral} provides preliminary results on the Young integral, Section \ref{section_examples} deals with some dynamics and observation schemes leading to Model (\ref{main_model_copies}), the aforementioned theoretical results on $\widehat b$ and $\widehat b_m'$ are established in Section \ref{section_nonparametric_estimators}, and Section \ref{section_numerical_experiments} shows that both $\widehat b$ and $\widehat b_m'$ are also satisfactory on the numerical side.
\\
\\
{\bf Notations:}
\begin{enumerate}
 \item For $E =\mathbb N$, $\mathbb R$ or $\mathbb R_+$, $E^* = E\backslash\{0\}$.
 \item The $\mathbb L^1$-norm (resp. $\mathbb L^2$-norm) is denoted by $\|.\|_1$ (resp. $\|.\|$), the uniform norm is denoted by $\|.\|_{\infty}$, and the $\alpha$-H\"older seminorm on $[0,T]$ is denoted by $\|.\|_{\alpha,T}$.
 \item Let $I\subset\mathbb R$ be an interval.
 \begin{itemize}
  \item $C^0(I;\mathbb R)$ is the space of the continuous functions from $I$ into $\mathbb R$.
  \item $C^1(I;\mathbb R)$ is the space of the continuously differentiable functions from $I$ into $\mathbb R$.
 \end{itemize}
\end{enumerate}
%


%
\section{Preliminaries on the Young integral}\label{section_Young_integral}
First, assume that the paths of the process $Z$ are locally $\alpha$-H\"older continuous, and let us define the pathwise integral with respect to $Z$ thanks to the Young integral.
%


%
\begin{definition}\label{Riemann_sums}
Consider $h,x\in C^0([0,T];\mathbb R)$, and let $D =\{t_0,\dots,t_n\}$ be a dissection of $[s,t]$ with $n\in\mathbb N^*$ and $s,t\in [0,T]$ satisfying $s < t$. The Riemann sum of $h$ with respect to $x$ along the dissection $D$ of $[s,t]$ is defined by
\begin{displaymath}
J_{h,x,D}(s,t) =
\sum_{i = 0}^{n - 1}h(t_i)(x(t_{i + 1}) - x(t_i)).
\end{displaymath}
\end{definition}
\noindent
{\bf Notation.} Let $D =\{t_0,\dots,t_n\}$ be a dissection of $[s,t]$ with $n\in\mathbb N^*$ and $s,t\in [0,T]$ satisfying $s < t$. Its mesh is denoted by $\pi(D)$:
\begin{displaymath}
\pi(D) =
\max_{i\in\{0,\dots,m - 1\}}
\{t_{i + 1} - t_i\}.
\end{displaymath}
%


%
\begin{theorem}\label{Young_integral}\index{Young's integral}
Consider $\beta\in (0,1]$ such that $\alpha +\beta > 1$, and let $h$ (resp. $x$) be a $\beta$-H\"older (resp. $\alpha$-H\"older) continuous function from $[0,T]$ into $\mathbb R$. Then, there exists a unique $\alpha$-H\"older continuous function $J_{h,x} : [0,T]\rightarrow\mathbb R$ such that, for any $s,t\in [0,T]$ satisfying $s < t$, and every sequence $(D_n)_{n\in\mathbb N}$ of dissections of $[s,t]$ satisfying $\pi(D_n)\rightarrow 0$ when $n\rightarrow\infty$,
\begin{displaymath}
\lim_{n\rightarrow\infty}
|J_{h,x}(t) - J_{h,x}(s) - J_{h,x,D_n}(s,t)| = 0.
\end{displaymath}
The Young integral on $[s,t]$ of $h$ with respect to $x$ is defined by
\begin{displaymath}
\int_{s}^{t}h(u)dx(u) =
J_{h,x}(t) - J_{h,x}(s).
\end{displaymath}
\end{theorem}
\noindent
See Friz and Victoir \cite{FV10}, Theorem 6.8 for a proof.
\\
\\
By Theorem \ref{Young_integral}, for every process $H$ which paths are locally $\beta$-H\"older continuous with $\beta\in (0,1]$ satisfying $\alpha +\beta > 1$, one may define the pathwise integral on $[0,T]$ of $H$ with respect to $Z$ by
\begin{displaymath}
\left(\int_{0}^{T}H_sdZ_s\right)(\omega) =
\int_{0}^{T}H_s(\omega)dZ_s(\omega)
\textrm{ $;$ }
\forall\omega\in\Omega.
\end{displaymath}
In particular, for every $h\in C^1([0,T];\mathbb R)$, since $\alpha + 1 > 1$,
\begin{displaymath}
\left(\int_{0}^{T}h(s)dZ_s\right)(\omega) =
\int_{0}^{T}h(s)dZ_s(\omega)
\textrm{ $;$ }
\forall\omega\in\Omega.
\end{displaymath}
Now, assume that $Z$ is a Gaussian process such that $\mathbb E(\|Z\|_{\alpha,T}^{2}) <\infty$, and which covariance function is denoted by $R$. For instance, if $Z$ is a fractional Brownian motion of Hurst parameter $H\in (0,1)$, then the paths of $Z$ are locally $\beta$-H\"older continuous and $\mathbb E(\|Z\|_{\beta,T}^{2}) <\infty$ for every $\beta\in (0,H)$ by the Garsia-Rodemich-Rumsey lemma (see Nualart \cite{NUALART06}, Lemma A.3.1 and - initially - Garsia et al. \cite{GRR70}). Assume also that $(Z^1,\dots,Z^N)$ is a Gaussian process whose components are copies of $Z$, and that there exists a correlation matrix $\Gamma_{\star} = (\Gamma_{\star}^{i,k})_{i,k}$ such that, for every $i,k\in\{1,\dots,N\}$ and $s,t\in [0,T]$,
\begin{equation}\label{noise_covariance_condition_Gaussian_preliminaries}
\mathbb E(Z_{s}^{i}Z_{t}^{k}) =\Gamma_{\star}^{i,k}R(s,t).
\end{equation}
%


%
\begin{definition}\label{2D_Riemann_sums}
Consider $\Lambda,\Theta\in C^0([0,T]^2;\mathbb R)$, and let $D =\{t_0,\dots,t_n\}$ be a dissection of $[0,T]$ with $n\in\mathbb N^*$. The Riemann sum of $\Lambda$ with respect to $\Theta$ along $D^2$ is defined by
\begin{displaymath}
\mathbb J_{\Lambda,\Theta,D^2} =
\sum_{j,\ell = 0}^{n - 1}\Lambda(t_j,t_{\ell})\Delta_{(t_j,t_{j + 1}),(t_{\ell},t_{\ell + 1})}\Theta,
\end{displaymath}
where
\begin{displaymath}
\Delta_{(s,t),(u,v)}\Theta =
\Theta(t,v) -\Theta(s,v) - (\Theta(t,u) -\Theta(s,u))
\textrm{ $;$ }\forall s,t,u,v\in [0,T].
\end{displaymath}
\end{definition}
%


%
\begin{proposition}\label{Young_Wiener_isometry}
Consider $\varphi,\psi\in C^1([0,T];\mathbb R)$.
\begin{enumerate}
 \item For every sequence $(D_n)_{n\in\mathbb N}$ of dissections of $[0,T]$ satisfying $\pi(D_n)\rightarrow 0$ when $n\rightarrow\infty$, the real sequence $(\mathbb J_{\varphi\otimes\psi,R,D_{n}^{2}})_{n\in\mathbb N}$ converges, and its limit is denoted by
 \begin{displaymath}
 \int_{[0,T]^2}\varphi(s)\psi(t)dR(s,t).
 \end{displaymath}
 \item The Young integral on $[0,T]$ of $\varphi$ (and $\psi$) with respect to $Z$ is a Gaussian random variable.
 \item For every $i,k\in\{1,\dots,N\}$,
 \begin{displaymath}
 \mathbb E\left[\left(\int_{0}^{T}\varphi(s)dZ_{s}^{i}\right)
 \left(\int_{0}^{T}\psi(s)dZ_{s}^{k}\right)\right] =
 \Gamma_{\star}^{i,k}\int_{[0,T]^2}\varphi(s)\psi(t)dR(s,t).
 \end{displaymath}
\end{enumerate}
\end{proposition}
%


%
\begin{proof}
First, since $\mathbb E(\|Z\|_{\alpha,T}^{2}) <\infty$, for every $s,t,u,v\in [0,T]$,
\begin{eqnarray*}
 |\Delta_{(s,t),(u,v)}R| & \leqslant &
 \mathbb E(|Z_t - Z_s|\cdot|Z_v - Z_u|)\\
 & \leqslant &
 \mathbb E(\|Z\|_{\alpha,T}^{2})|t - s|^{\alpha}|v - u|^{\alpha},
\end{eqnarray*}
and then $R$ is a continuous function of finite 2D $1/\alpha$-variation from $[0,T]^2$ into $\mathbb R$. Since the functions $\varphi$ and $\psi$ are continuously differentiable - and then Lipschitz continuous - from $[0,T]$ into $\mathbb R$, Proposition \ref{Young_Wiener_isometry}.(1) is a straightforward consequence of Friz and Victoir \cite{FV10}, Theorem 6.18. Now, by Friz and Victoir \cite{FV10}, Proposition 15.39, the Young integral on $[0,T]$ of $\varphi$ (and $\psi$) with respect to $Z$ is a Gaussian random variable, and
\begin{equation}\label{Young_Wiener_isometry_1}
\mathbb E\left[\left(\int_{0}^{T}\varphi(s)dZ_s\right)
\left(\int_{0}^{T}\psi(s)dZ_s\right)\right] =
\int_{[0,T]^2}\varphi(s)\psi(t)dR(s,t).
\end{equation}
Finally, for any $i,k\in\{1,\dots,N\}$, thanks to the integration by parts formula, and by Equality (\ref{noise_covariance_condition_Gaussian_preliminaries}),
\begin{eqnarray*}
 \mathbb E\left[\left(\int_{0}^{T}\varphi(s)dZ_{s}^{i}\right)
 \left(\int_{0}^{T}\psi(s)dZ_{s}^{k}\right)\right] & = &
 \mathbb E\left[\left(\varphi(T)Z_{T}^{i} -\int_{0}^{T}\varphi'(s)Z_{s}^{i}ds\right)
 \left(\psi(T)Z_{T}^{k} -\int_{0}^{T}\psi'(s)Z_{s}^{k}ds\right)\right]\\
 & = &
 \Gamma_{\star}^{i,k}
 \left(\varphi(T)\psi(T)R(T,T) -\varphi(T)\int_{0}^{T}\psi'(s)R(s,T)ds\right.\\
 & &
 \hspace{0.75cm}\left.
 -\psi(T)\int_{0}^{T}\varphi'(s)R(s,T)ds +
 \int_{[0,T]^2}\varphi'(s)\psi'(s)R(s,t)dsdt\right)\\
 & = &
 \Gamma_{\star}^{i,k}
 \mathbb E\left[\left(\int_{0}^{T}\varphi(s)dZ_s\right)
 \left(\int_{0}^{T}\psi(s)dZ_s\right)\right],
\end{eqnarray*}
leading to
\begin{displaymath}
\mathbb E\left[\left(\int_{0}^{T}\varphi(s)dZ_{s}^{i}\right)
\left(\int_{0}^{T}\psi(s)dZ_{s}^{k}\right)\right] =
\Gamma_{\star}^{i,k}\int_{[0,T]^2}\varphi(s)\psi(t)dR(s,t)
\quad\textrm{by Equality (\ref{Young_Wiener_isometry_1})}.
\end{displaymath}
\end{proof}
%


%
\section{Some dynamics and observation schemes leading to Model (\ref{main_model_copies})}\label{section_examples}
%


%
\subsection{From (fractional) diffusions and interacting particle systems to Model (\ref{main_model_copies})}\label{subsection_examples_dynamics}
This section shows how to construct $X^1,\dots,X^N$ from copies of a linear (fractional) diffusion, and from a basic interacting particle system. Moreover, possible applications in finance and in pharmacokinetics are mentioned.
%


%
\subsubsection{Copies of a linear (fractional) diffusion}\label{subsubsection_copies_linear_diffusion}
First, for any $i\in\{1,\dots,N\}$, consider the process $S^i$ defined by the (It\^o) stochastic differential equation
\begin{equation}\label{linear_SDE_correlated_noises}
S_{t}^{i} = S_0 +\int_{0}^{t}\texttt b_0(u)S_{u}^{i}du +
\sigma\int_{0}^{t}S_{u}^{i}{\rm d}W_{u}^{i}
\textrm{ $;$ }t\in [0,T],
\end{equation}
where $S_0\in\mathbb R$, $\sigma\in\mathbb R^*$, $\texttt b_0\in C^0(\mathbb R_+;\mathbb R)$, $W^1,\dots,W^N$ are Brownian motions such that $d\langle W^i,W^k\rangle_t = R^{i,k}dt$ for every $i,k\in\{1,\dots,N\}$, and $R = (R^{i,k})_{i,k}$ is a correlation matrix. For instance, (\ref{linear_SDE_correlated_noises}) is appropriate to model the prices of $N$ interacting risky assets of the same kind as in Duellmann et al. \cite{DKK10}. For every $t\in [0,T]$, thanks to the It\^o formula,
\begin{equation}\label{linear_SDE_correlated_noises_transformed}
S_{t}^{i} = S_0\exp(X_{t}^{i}),\quad
X_{t}^{i} = b_0(t) + Z_{t}^{i},\quad
b_0(t) =\int_{0}^{t}\left(\texttt b_0(u) -\frac{\sigma^2}{2}\right)du
\quad {\rm and}\quad
Z_{t}^{i} =\sigma W_{t}^{i}.
\end{equation}
Moreover, $Z^1,\dots,Z^N$ are centered continuous processes such that, by the stochastic integration by parts formula,
\begin{displaymath}
T\sup_{t\in [0,T]}|\mathbb E(Z_{t}^{i}Z_{t}^{k})| =
\sigma^2T\sup_{t\in [0,T]}|\langle W^i,W^k\rangle_t| =
(\sigma T)^2|R^{i,k}|,
\end{displaymath}
leading to
\begin{displaymath}
\left|\int_{0}^{T}\mathbb E(Z_{s}^{i}Z_{s}^{k})ds\right|
\leqslant
(\sigma T)^2|R^{i,k}|.
\end{displaymath}
Thus, $Z^1,\dots,Z^N$ satisfy the condition (\ref{noise_covariance_condition}) with
\begin{displaymath}
\Gamma^{i,k} = (\sigma T)^2|R^{i,k}|
\textrm{ $;$ }i,k\in\{1,\dots,N\}.
\end{displaymath}
Now, for any $i\in\{1,\dots,N\}$, consider the process $C^i$ defined by the pathwise (Young) differential equation
\begin{equation}\label{linear_fDE_random_effect}
C_{t}^{i} = C_0 +\int_{0}^{t}(\texttt b_0(s) +\phi^i)C_{s}^{i}ds +
\sigma\int_{0}^{t}C_{s}^{i}dB_{s}^{i}
\textrm{ $;$ }t\in [0,T],
\end{equation}
where $C_0\in\mathbb R$, $\sigma\in\mathbb R^*$, $\texttt b_0\in C^0(\mathbb R_+;\mathbb R)$, $B^1,\dots,B^N$ are independent fractional Brownian motions of Hurst parameter $H\in (1/2,1)$, and $\phi^1,\dots,\phi^N$ are i.i.d. centered square-integrable random variables of common variance $\sigma_{\phi}^{2} > 0$ such that $(\phi^1,\dots,\phi^N)$ and $(B^1,\dots,B^N)$ are independent. Such a model with random effects is commonly used in population pharmacokinetics (see Donnet and Samson \cite{DS13}), but usually $H = 1/2$ and the stochastic integral in Equation (\ref{linear_fDE_random_effect}) is taken in the sense of It\^o. However, to take $H$ close to $1$ ensures that the paths of $C^i$ are less irregular than those of a diffusion process, which is appropriate to model the elimination process of a drug by the $i$-th patient of a clinical trial. For every $t\in [0,T]$, thanks to the change of variable formula for Young's integral,
\begin{equation}\label{linear_fDE_random_effect_transformed}
C_{t}^{i} = C_0\exp(X_{t}^{i}),\quad
X_{t}^{i} = b_0(t) + Z_{t}^{i},\quad
b_0(t) =\int_{0}^{t}\texttt b_0(s)ds
\quad {\rm and}\quad
Z_{t}^{i} =\phi^it +\sigma B_{t}^{i}.
\end{equation}
Moreover, $Z^1,\dots,Z^N$ are i.i.d. centered continuous processes such that
\begin{displaymath}
T\sup_{t\in [0,T]}\mathbb E(|Z_{t}^{i}|^2) =
T(\sigma_{\phi}^{2}T^2 +\sigma^2T^{2H}) =:\Sigma^2.
\end{displaymath}
Thus, $Z^1,\dots,Z^N$ satisfy the condition (\ref{noise_covariance_condition}) with
\begin{displaymath}
\Gamma^{i,k} =
\Sigma^2\mathbf 1_{i = k}
\textrm{ $;$ }i,k\in\{1,\dots,N\}.
\end{displaymath}
%


%
\subsubsection{A basic interacting particle system}\label{subsubsection_interacting_particle_system}
For any $i\in\{1,\dots,N\}$, consider the process $Y^i$ defined by
\begin{equation}\label{interacting_particle_system}
Y_{t}^{i} = Y_0 +\int_{0}^{t}(\texttt b_0'(s) - (Y_{s}^{i} -\overline Y_s))ds +
\sigma W_{t}^{i}
\textrm{ $;$ }t\in [0,T],
\end{equation}
where $Y_0\in\mathbb R$, $\sigma\in\mathbb R^*$, $\texttt b_0\in C^1(\mathbb R_+;\mathbb R)$ with $\texttt b_0(0) = 0$, and $W^1,\dots,W^N$ are independent Brownian motions. Moreover, let $X^i$ be the process defined by
\begin{equation}\label{interacting_particle_system_transformed}
X_{t}^{i} = Y_{t}^{i} +\int_{0}^{t}Y_{s}^{i}ds - Y_0(1 + t)
\textrm{ $;$ }t\in [0,T].
\end{equation}
For every $t\in [0,T]$,
\begin{displaymath}
\overline Y_t = Y_0 +\texttt b_0(t) +\frac{\sigma}{N}\sum_{\ell = 1}^{N}W_{t}^{\ell},
\end{displaymath}
leading to
\begin{eqnarray*}
 X_{t}^{i} & = &
 Y_0 - Y_0(1 + t) +
 \int_{0}^{t}(\texttt b_0'(s) - (Y_{s}^{i} -\overline Y_s))ds +
 \int_{0}^{t}Y_{s}^{i}ds +\sigma W_{t}^{i}\\
 & = &
 -Y_0t +\texttt b_0(t) +\int_{0}^{t}\left(
 Y_0 +\texttt b_0(t) +\frac{\sigma}{N}\sum_{\ell = 1}^{N}W_{t}^{\ell}\right)ds +
 \sigma W_{t}^{i} = b_0(t) + Z_{t}^{i}
\end{eqnarray*}
with
\begin{displaymath}
b_0(t) =\texttt b_0(t) +\int_{0}^{t}\texttt b_0(s)ds
\quad {\rm and}\quad
Z_{t}^{i} =
\sigma\left(W_{t}^{i} +\frac{1}{N}\sum_{\ell = 1}^{N}\int_{0}^{t}W_{s}^{\ell}ds\right).
\end{displaymath}
The following proposition provides a matrix $\overline\Gamma = (\overline\Gamma^{i,k})_{i,k}$ satisfying the condition (\ref{noise_covariance_condition}) for $Z^1,\dots,Z^N$.
%


%
\begin{proposition}\label{matrix_Gamma_interacting_particle_system}
For every $i,k\in\{1,\dots,N\}$,
\begin{displaymath}
T\sup_{t\in [0,T]}|\mathbb E(Z_{t}^{i}Z_{t}^{k})|
\leqslant
\overline\Gamma^{i,k},
\end{displaymath}
where
\begin{displaymath}
\overline\Gamma^{i,k} =
T(1\vee T)^3\sigma^2\left(\mathbf 1_{i = k} +\frac{3}{N}\right).
\end{displaymath}
\end{proposition}
%


%
\begin{proof}
Consider $i,k\in\{1,\dots,N\}$ and $t\in [0,T]$. Since $W^1,\dots,W^N$ are independent Brownian motions,
\begin{eqnarray*}
 \frac{1}{\sigma^2}\mathbb E(Z_{t}^{i}Z_{t}^{k}) & = &
 \mathbb E\left[
 \left(W_{t}^{i} +\frac{1}{N}\sum_{\ell = 1}^{N}\int_{0}^{t}W_{s}^{\ell}ds\right)
 \left(W_{t}^{k} +\frac{1}{N}\sum_{\ell = 1}^{N}\int_{0}^{t}W_{s}^{\ell}ds\right)\right]\\
 & = &
 \mathbb E(W_{t}^{i}W_{t}^{k}) +
 \frac{1}{N}\sum_{\ell = 1}^{N}\int_{0}^{t}\mathbb E(W_{t}^{k}W_{s}^{\ell})ds\\
 & &
 \hspace{3cm} +
 \frac{1}{N}\sum_{\ell = 1}^{N}\int_{0}^{t}\mathbb E(W_{t}^{i}W_{s}^{\ell})ds +
 \frac{1}{N^2}\sum_{\ell,\ell' = 1}^{N}\int_{[0,t]^2}
 \mathbb E(W_{s}^{\ell}W_{s'}^{\ell'})dsds'\\
 & = &
 t\mathbf 1_{i = k} +
 \frac{2}{N}\int_{0}^{t}(s\wedge t)ds +
 \frac{1}{N}\int_{0}^{t}\int_{0}^{t}(s\wedge s')dsds'.
\end{eqnarray*}
Therefore,
\begin{displaymath}
\sup_{t\in [0,T]}|\mathbb E(Z_{t}^{i}Z_{t}^{k})|
\leqslant
(1\vee T)^3\sigma^2\left(\mathbf 1_{i = k} +\frac{3}{N}\right).
\end{displaymath}
\end{proof}
%


%
\subsection{From long-time observation to correlated copies}\label{subsection_examples_observation_scheme}
In this section, $Z$ is a fractional Brownian motion of Hurst parameter $H\in (0,1)$, and a single observation of $X$ on $\mathbb R_+$ is available. For every $i\in\{1,\dots,N\}$, consider $t_i(\Delta) = (i - 1)(T +\Delta)$ with $\Delta > 0$, and
\begin{equation}\label{long_time_observation_copies_fractional_model}
X_{t}^{i} = X_{t_i(\Delta) + t} - X_{t_i(\Delta)}
\textrm{ $;$ }t\in [0,T].
\end{equation}
Assume that $b_0$ is a $(T +\Delta)$-periodic function such that $b_0(0) = 0$. Then, for every $t\in [0,T]$,
\begin{eqnarray*}
 X_{t}^{i} & = &
 b_0(t_i(\Delta) + t) - b_0(t_i(\Delta)) + Z_{t_i(\Delta) + t} - Z_{t_i(\Delta)}\\
 & = &
 b_0(t) + Z_{t}^{i}
 \quad {\rm with}\quad
 Z_{t}^{i} =
 Z_{t_i(\Delta) + t} - Z_{t_i(\Delta)}.
\end{eqnarray*}
Since the fractional Brownian motion has stationary increments, $Z^1,\dots,Z^N$ are correlated copies of $Z$, and then $X^1,\dots,X^N$ are correlated copies of $X$. First, the following proposition provides a matrix $(\Gamma^{i,k}(H,\Delta))_{i,k}$ satisfying the condition (\ref{noise_covariance_condition}) for these copies $Z^1,\dots,Z^N$ of $Z$.
%


%
\begin{proposition}\label{matrix_Gamma_long_time_drifted_fBm}
Assume that $\Delta\geqslant T$. For every $i,k\in\{1,\dots,N\}$,
\begin{displaymath}
T\sup_{s\in [0,T]}|\mathbb E(Z_{s}^{i}Z_{s}^{k})|
\leqslant
\Gamma^{i,k}(H,\Delta),
\end{displaymath}
where
\begin{displaymath}
\Gamma^{i,k}(H,\Delta) =
\mathfrak c_{\ref{matrix_Gamma_long_time_drifted_fBm},1}(H)\mathbf 1_{i = k} +
\mathfrak c_{\ref{matrix_Gamma_long_time_drifted_fBm},2}(H)|k - i|^{2H - 2}\Delta^{2H - 2}\mathbf 1_{i\neq k}
\end{displaymath}
with
\begin{displaymath}
\mathfrak c_{\ref{matrix_Gamma_long_time_drifted_fBm},1}(H) =
T^{2H + 1}
\quad\textrm{and}\quad
\mathfrak c_{\ref{matrix_Gamma_long_time_drifted_fBm},2}(H) =
4H|2H - 1|T^3.
\end{displaymath}
\end{proposition}
%


%
\begin{proof}
Consider $s\in [0,T]$ and, without loss of generality, $i,k\in\{1,\dots,N\}$ such that $k\geqslant i$. First, since the fractional Brownian motion has stationary increments,
\begin{eqnarray*}
 \mathbb E(Z_{s}^{i}Z_{s}^{k}) & = &
 \mathbb E[(Z_{t_i(\Delta) + s} - Z_{t_i(\Delta)})(Z_{t_k(\Delta) + s} - Z_{t_k(\Delta)})]\\
 & = &
 \mathbb E[Z_s(Z_{(k - i)(T +\Delta) + s} - Z_{(k - i)(T +\Delta)})] =
 s^{2H}\mathbf 1_{i = k} +
 \frac{1}{2}R_{s}^{i,k}\mathbf 1_{k > i},
\end{eqnarray*}
where
\begin{displaymath}
R_{s}^{i,k} =
[(k - i)(T +\Delta)]^{2H}\left[
\left(1 +\frac{s}{(k - i)(T +\Delta)}\right)^{2H} +
\left(1 -\frac{s}{(k - i)(T +\Delta)}\right)^{2H} - 2\right]
\textrm{ $;$ }k > i.
\end{displaymath}
Now, assume that $k > i$. By the Taylor-Lagrange formula, there exists
\begin{displaymath}
\xi_{\pm s}^{i,k}\in\left(1,1\pm\frac{s}{(k - i)(T +\Delta)}\right)\subset (0,\infty)
\end{displaymath}
such that
\begin{displaymath}
\left(1\pm\frac{s}{(k - i)(T +\Delta)}\right)^{2H} =
1\pm\frac{2Hs}{(k - i)(T +\Delta)} +
H(2H - 1)\left(\frac{s}{(k - i)(T +\Delta)}\right)^2(\xi_{\pm s}^{i,k})^{2H - 2}.
\end{displaymath}
Then,
\begin{displaymath}
R_{s}^{i,k} =
H(2H - 1)\cdot [(k - i)(T +\Delta)]^{2H - 2}\cdot
s^2[(\xi_{s}^{i,k})^{2H - 2} + (\xi_{-s}^{i,k})^{2H - 2}].
\end{displaymath}
Moreover, since $2H - 2 < 0$,
\begin{displaymath}
(T +\Delta)^{2H - 2}\leqslant\Delta^{2H - 2}
\end{displaymath}
and
\begin{eqnarray*}
 (\xi_{\pm s}^{i,k})^{2H - 2}
 & \leqslant &
 \left(1 -\frac{T}{(k - i)(T +\Delta)}\right)^{2H - 2}\\
 & \leqslant &
 \left(\frac{\Delta}{T +\Delta}\right)^{2H - 2}\leqslant
 \left(\frac{1}{2}\right)^{2H - 2}\leqslant 4.
\end{eqnarray*}
Thus,
\begin{displaymath}
|R_{s}^{i,k}|
\leqslant
8H|2H - 1|T^2|k - i|^{2H - 2}\Delta^{2H - 2}.
\end{displaymath}
In conclusion, for every $i,k\in\{1,\dots,N\}$ and $s\in [0,T]$,
\begin{displaymath}
|\mathbb E(Z_{s}^{i}Z_{s}^{k})|
\leqslant
T^{2H}\mathbf 1_{i = k} +
4H|2H - 1|T^2|k - i|^{2H - 2}\Delta^{2H - 2}\mathbf 1_{i\neq k}.
\end{displaymath}
\end{proof}
\noindent
Now, as in Marie \cite{MARIE23}, let us consider a financial market model in which the prices process $S = (S_t)_{t\in\mathbb R_+}$ and the volatility process $\Sigma = (\Sigma_t)_{t\in\mathbb R_+}$ of the risky asset satisfy
\begin{equation}\label{Wiggins_fractional}
\left\{
\begin{array}{rcll}
 {\rm d}S_t & = & \mu(t)S_tdt +\Sigma_tS_t{\rm d}W_t & \quad ({\rm A})\\
 d\Sigma_t & = & \texttt b_0(t)\Sigma_tdt +\nu\Sigma_tdB_t & \quad ({\rm B})
\end{array}\right.,
\end{equation}
where $S_0,\Sigma_0 > 0$, $W = (W_t)_{t\in\mathbb R_+}$ (resp. $B = (B_t)_{t\in\mathbb R_+}$) is a Brownian motion (resp. a fractional Brownian motion of Hurst parameter $H\in (1/2,1)$), $W$ and $B$ are independent, $\nu > 0$ and $\mu,\texttt b_0\in C^0(\mathbb R_+;\mathbb R)$. Since the paths of $B$ are locally $\alpha$-H\"older continuous for every $\alpha\in (1/2,H)$, Equation (B) has a unique pathwise solution $\Sigma$ which is adapted to the natural filtration $\mathbb F$ of $(W,B)$, and since $W$ and $B$ are independent Gaussian processes, $W$ is a $\mathbb F$-Brownian motion, leading to the existence and uniqueness of the solution of Equation (A) in the sense of It\^o. Moreover, since $H > 1/2$, Model (\ref{Wiggins_fractional}) takes into account the persistence-in-volatility phenomenon as in Comte et al. \cite{CCR12}. For every $t\in\mathbb R_+$, thanks to the change of variable formula for Young's integral,
\begin{displaymath}
\Sigma_t =\Sigma_0\exp(X_t),\quad
X_t = b_0(t) + Z_t,\quad
b_0(t) =\int_{0}^{t}\texttt b_0(s)ds\quad {\rm and}\quad
Z_t =\nu B_t.
\end{displaymath}
Usually, only one long-time observation of $\Sigma$, and then of $X$, are available. So, in order to consider the estimators provided in Section \ref{section_nonparametric_estimators} for Model (\ref{main_model_copies}), $X^1,\dots,X^N$ must be defined from $(X_t)_{t\in\mathbb R_+}$ by (\ref{long_time_observation_copies_fractional_model}).
%


%
\section{Nonparametric estimators of $b_0$ and $b_0'$}\label{section_nonparametric_estimators}
%


%
\subsection{Nonparametric estimator of $b_0$}\label{subsection_estimator_b_0}
The following proposition provides a risk bound on $\widehat b =\overline X$.
%


%
\begin{proposition}\label{risk_bound_estimator_b_0}
Assume that $Z^1,\dots,Z^N$ satisfy the condition (\ref{noise_covariance_condition}). Then,
\begin{equation}\label{risk_bound_estimator_b_0_1}
\mathbb E(\|\widehat b - b_0\|^2)\leqslant\mathfrak R_N
\quad\textrm{with}\quad
\mathfrak R_N =\frac{1}{N^2}\sum_{i,k = 1}^{N}\Gamma^{i,k}.
\end{equation}
\end{proposition}
%


%
\begin{proof}
Since $Z^1,\dots,Z^N$ satisfy the condition (\ref{noise_covariance_condition}),
\begin{eqnarray*}
 \mathbb E(\|\widehat b - b_0\|^2)
 & = &
 \frac{1}{N^2}\mathbb E\left(\left\|\sum_{i = 1}^{N}Z^i\right\|^2\right)\\
 & = &
 \frac{1}{N^2}\sum_{i,k = 1}^{N}\int_{0}^{T}\mathbb E(Z_{s}^{i}Z_{s}^{k})ds\leqslant
 \frac{1}{N^2}\sum_{i,k = 1}^{N}\Gamma^{i,k}.
\end{eqnarray*}
\end{proof}
%


%
\begin{remark}\label{remark_risk_bound_estimator_b_0}
Assume that
\begin{displaymath}
\sigma^2 =\int_{0}^{T}\mathbb E(Z_{s}^{2})ds <\infty,
\end{displaymath}
and that $Z^1,\dots,Z^N$ are independent copies of $Z$. Then, one may take $\Gamma^{i,k} = 0$ for every $i,k\in\{1,\dots,N\}$ such that $i\not= k$, and $\Gamma^{i,i} =\sigma^2$ for every $i\in\{1,\dots,N\}$, leading to
\begin{displaymath}
\mathbb E(\|\widehat b - b_0\|^2)
\leqslant\frac{\sigma^2}{N}
\quad\textrm{by Proposition \ref{risk_bound_estimator_b_0}}.
\end{displaymath}
\end{remark}
%


%
\subsection{Nonparametric estimator of $b_0'$}\label{subsection_estimator_b_0'}
Throughout this section, $b_0$ is continuously differentiable from $[0,T]$ into $\mathbb R$, and the paths of $Z$ are locally $\alpha$-H\"older continuous with $\alpha\in (0,1]$. In Models (\ref{linear_SDE_correlated_noises_transformed}) and (\ref{linear_fDE_random_effect_transformed}) presented in Section \ref{subsection_examples_dynamics}, $b_0$ depends on $\int_{0}^{.}\texttt b_0$ with $\texttt b_0\in C^0([0,T];\mathbb R)$. For this reason, $b_0'$ needs to be estimated. As explained in the introduction section, a natural estimator of $b_0'$ is given by
\begin{displaymath}
\widehat b_m' =\sum_{j = 1}^{m}
\mathcal I(\varphi_j,\overline X)\varphi_j,
\end{displaymath}
where $\mathcal I(\cdot)$ is the Young integral operator.
\\
\\
The outline of this section is as follows: Proposition \ref{risk_bound_estimator_b_0'_general} provides a risk bound on $\widehat b_m'$ with no additional condition on the $Z^i$'s and then, when $Z$ is a Gaussian process with a regular enough covariance function, thanks to the Young-Wiener isometry (see Proposition \ref{Young_Wiener_isometry}.(3)); Proposition \ref{risk_bound_estimator_b_0'_Gaussian} deals with a sharper risk bound on $\widehat b_m'$, and a model selection procedure is provided with theoretical guarantees on the corresponding adaptive estimator in Theorem \ref{risk_bound_adaptive_estimator_b_0'_Gaussian}.
\\
\\
{\bf Notation.} In the sequel,
\begin{displaymath}
\mathcal L(m) =
\sum_{j = 1}^{m}\|\varphi_j\|_{\infty}^{2}
\quad {\rm and}\quad
\overline{\mathcal L}(m) =
\sum_{j = 1}^{m}\|\varphi_j'\|_{\infty}^{2}.
\end{displaymath}
First, the following proposition provides a risk bound on $\widehat b_m'$ in the general case.
%


%
\begin{proposition}\label{risk_bound_estimator_b_0'_general}
Assume that for every $i,k\in\{1,\dots,N\}$,
\begin{equation}\label{risk_bound_estimator_b_0'_general_1}
T\sup_{t\in [0,T]}|\mathbb E(Z_{t}^{i}Z_{t}^{k})|
\leqslant\Gamma^{i,k}
\quad\textrm{with}\quad
\Gamma^{i,k}\geqslant 0,
\quad\textrm{and}\quad
\sum_{i,k = 1}^{N}\Gamma^{i,k}\underset{N\rightarrow\infty}{=}{\rm o}(N^2).
\end{equation}
Then,
\begin{equation}\label{risk_bound_estimator_b_0'_general_2}
\mathbb E(\|\widehat b_m' - b_0'\|^2)\leqslant
\min_{\tau\in\mathcal S_m}\|\tau - b_0'\|^2 +
2(T\vee T^{-1})
(\mathcal L(m) +\overline{\mathcal L}(m))\mathfrak R_N.
\end{equation}
\end{proposition}
%


%
\begin{proof}
First, since $\widehat b_m' -\Pi_m(b_0')\in\mathcal S_m$,
\begin{eqnarray*}
 \|\widehat b_m' - b_0'\|^2 & = &
 \|\widehat b_m' -\Pi_m(b_0') +\Pi_m(b_0') - b_0'\|^2\\
 & = &
 \min_{\tau\in\mathcal S_m}\|\tau - b_0'\|^2 +
 \|\widehat b_m' -\Pi_m(b_0')\|^2.
\end{eqnarray*}
Moreover,
\begin{eqnarray*}
 \|\widehat b_m' -\Pi_m(b_0')\|^2 & = &
 \left\|\sum_{j = 1}^{m}(\mathcal I(\varphi_j,\overline X) -\langle\varphi_j,b_0'\rangle)\varphi_j\right\|^2\\
 & = &
 \left\|\sum_{j = 1}^{m}\mathcal I(\varphi_j,\overline X - b_0)\varphi_j\right\|^2 =
 \underbrace{\sum_{j = 1}^{m}\left(\int_{0}^{T}\varphi_j(s)d\overline Z_s\right)^2}_{=: V_m}.
\end{eqnarray*}
Then,
\begin{equation}\label{risk_bound_estimator_b_0'_general_3}
\|\widehat b_m' - b_0'\|^2 =
\min_{\tau\in\mathcal S_m}\|\tau - b_0'\|^2 + V_m.
\end{equation}
Now, thanks to the integration by parts formula,
\begin{displaymath}
\int_{0}^{T}\varphi_j(s)d\overline Z_s =
\varphi_j(T)\overline Z_T -\langle\varphi_j',\overline Z\rangle
\textrm{ $;$ }\forall j\in\{1,\dots,m\},
\end{displaymath}
leading to
\begin{eqnarray*}
 \mathbb E(V_m) & = &
 \sum_{j = 1}^{m}\mathbb E(|\varphi_j(T)\overline Z_T -\langle\varphi_j',\overline Z\rangle|^2)\\
 & \leqslant &
 2\mathbb E(\overline Z_{T}^{2})\sum_{j = 1}^{m}\varphi_j(T)^2 +
 2\mathbb E(\|\overline Z\|^2)\sum_{j = 1}^{m}\|\varphi_j'\|^2\leqslant
 \frac{2(\mathcal L(m)T^{-1} +\overline{\mathcal L}(m)T)}{N^2}\sum_{i,k = 1}^{N}\Gamma^{i,k}.
\end{eqnarray*}
Therefore, by Equality (\ref{risk_bound_estimator_b_0'_general_3}),
\begin{displaymath}
\mathbb E(\|\widehat b_m' - b_0'\|^2)\leqslant
\min_{\tau\in\mathcal S_m}\|\tau - b_0'\|^2 +
2(T\vee T^{-1})\frac{\mathcal L(m) +\overline{\mathcal L}(m)}{N^2}\sum_{i,k = 1}^{N}\Gamma^{i,k}.
\end{displaymath}
\end{proof}
%


%
\begin{remark}\label{remark_risk_bound_estimator_b_0'_general}
Let us make a few comments about the risk bound (\ref{risk_bound_estimator_b_0'_general_2}).
\begin{enumerate}
 \item The order of the bias term
 \begin{displaymath}
 \min_{\tau\in\mathcal S_m}\|\tau - b_0'\|^2
 \quad\textrm{in Proposition \ref{risk_bound_estimator_b_0'_general} depends on the $\varphi_j$'s}.
 \end{displaymath}
 For instance, assume that $(\varphi_1,\dots,\varphi_m)$ is the $[0,T]$-supported trigonometric basis: $\varphi_1 = T^{-\frac{1}{2}}\mathbf 1_{[0,T]}$ and, for every $j\in\mathbb N^*$ such that $2j + 1\leqslant m$,
 \begin{displaymath}
 \varphi_{2j} =\sqrt{\frac{2}{T}}\cos\left(2\pi j\frac{\cdot}{T}\right)
 \quad\textrm{and}\quad
 \varphi_{2j + 1} =\sqrt{\frac{2}{T}}\sin\left(2\pi j\frac{\cdot}{T}\right).
 \end{displaymath}
 Assume also that $b_0$ belongs to the Fourier-Sobolev space
 \begin{displaymath}
 \mathbb W_{2}^{\beta}([0,T]) =
 \left\{\varphi : [0,T]\rightarrow\mathbb R
 \textrm{ $\beta$ times differentiable} :
 \int_{0}^{T}\varphi^{(\beta)}(x)^2dx <\infty\right\}
 \textrm{ $;$ }\beta\in\mathbb N^*.
 \end{displaymath}
 By DeVore and Lorentz \cite{DL93}, Corollary 2.4 p. 205, there exists a constant $\mathfrak c_{\beta} > 0$, not depending on $m$, such that
 \begin{displaymath}
 \min_{\tau\in\mathcal S_m}\|\tau - b_0\|^2\leqslant\mathfrak c_{\beta}m^{-2\beta}.
 \end{displaymath}
 Moreover,
 \begin{displaymath}
 \mathcal L(m)\lesssim m
 \quad\textrm{and}\quad
 \overline{\mathcal L}(m)\lesssim m^3.
 \end{displaymath}
 So,
 \begin{displaymath}
 \mathbb E(\|\widehat b_m' - b_0'\|^2)\lesssim
 m^{-2\beta} + m^3\mathfrak R_N
 \quad\textrm{by Proposition \ref{risk_bound_estimator_b_0'_general}},
 \end{displaymath}
 and then $\widehat b_m'$ reaches the bias-variance tradeoff for
 \begin{displaymath}
 m\asymp\mathfrak R_{N}^{-\frac{1}{2\beta + 3}},
 \quad\textrm{leading to the rate $\mathfrak R_{N}^{\frac{2\beta}{2\beta + 3}}$}.
 \end{displaymath}
 \item When $Z^1,\dots,Z^N$ are independent copies of $Z$, at least for the trigonometric basis, the risk bound (\ref{risk_bound_estimator_b_0'_general_2}) on $\widehat b_m'$ is of same order as that on the estimator of the regression function derivative in Comte and Marie \cite{CM23} (see Proposition 3.9 and Section 4.2). However, the second part of Section \ref{subsection_estimator_b_0'} deals with a sharper risk bound on $\widehat b_m'$ when $Z$ is a Gaussian process with a regular enough covariance function.
\end{enumerate}
\end{remark}
\noindent
Now, assume that $Z$ is a Gaussian process such that $\mathbb E(\|Z\|_{\alpha,T}^{2}) <\infty$, and which covariance function is denoted by $R$. Moreover, there exists a constant $\mathfrak c_R > 0$ such that, for every $\varphi,\psi\in C^1([0,T];\mathbb R)$,
\begin{equation}\label{noise_covariance_condition_Gaussian_1}
|\langle\varphi,\psi\rangle_R|
\leqslant\mathfrak c_R\|\varphi\|\cdot\|\psi\|
\quad {\rm with}\quad
\langle\varphi,\psi\rangle_R :=
\int_{[0,T]^2}\varphi(s)\psi(t)dR(s,t).
\end{equation}
Assume also that $(Z^1,\dots,Z^N)$ is a Gaussian process whose components are copies of $Z$, and that there exists a correlation matrix $\Gamma_{\star} = (\Gamma_{\star}^{i,k})_{i,k}$ such that, for every $i,k\in\{1,\dots,N\}$ and $s,t\in [0,T]$,
\begin{equation}\label{noise_covariance_condition_Gaussian_2}
\mathbb E(Z_{s}^{i}Z_{t}^{k}) =\Gamma_{\star}^{i,k}R(s,t)
\quad\textrm{with}\quad
\sum_{i,k = 1}^{N}|\Gamma_{\star}^{i,k}|\underset{N\rightarrow\infty}{=}\textrm o(N^2).
\end{equation}
In particular, $Z^1,\dots,Z^N$ satisfy the condition (\ref{noise_covariance_condition}) with
\begin{equation}\label{matrix_Gamma_Gaussian}
\Gamma^{i,k} =
\underbrace{\left|\int_{0}^{T}R(s,s)ds\right|}_{=:\gamma_R}\cdot
|\Gamma_{\star}^{i,k}|
\textrm{ $;$ }\forall i,k\in\{1,\dots,N\}.
\end{equation}
In both the proofs of Proposition \ref{risk_bound_estimator_b_0'_Gaussian} and Theorem \ref{risk_bound_adaptive_estimator_b_0'_Gaussian}, the condition (\ref{noise_covariance_condition_Gaussian_2}) allows to apply Proposition \ref{Young_Wiener_isometry}.(3) to establish controls depending on $\|.\|_R$, and then the condition (\ref{noise_covariance_condition_Gaussian_1}) allows to use that $(\varphi_1,\dots,\varphi_N)$ is an orthonormal family of $\mathbb L^2([0,T],dt)$.
%


%
\begin{example}\label{example_noise_covariance_condition_Gaussian_1}
Let us provide a few examples of Gaussian processes satisfying the condition (\ref{noise_covariance_condition_Gaussian_1}).
\begin{enumerate}
 \item Assume that $Z$ is a Brownian motion. Then, for every $\varphi,\psi\in C^1([0,T];\mathbb R)$,
 \begin{eqnarray*}
  |\langle\varphi,\psi\rangle_R| & = &
  \left|\mathbb E\left(\left(\int_{0}^{T}\varphi(s)dZ_s\right)
  \left(\int_{0}^{T}\psi(s)dZ_s\right)\right)\right|
  \quad\textrm{(by Proposition \ref{Young_Wiener_isometry})}\\
  & = &
  \left|\int_{0}^{T}\varphi(s)\psi(s)d\langle Z\rangle_s\right|
  \leqslant
  \mathfrak c_R\|\varphi\|\cdot\|\psi\|
  \quad\textrm{with}\quad
  \mathfrak c_R = 1.
 \end{eqnarray*}
 \item Consider $U_T =\{(s,t)\in [0,T]^2 : s\neq t\}$, and assume that
 \begin{equation}\label{example_noise_covariance_condition_Gaussian_1_1}
 dR(s,t) =\rho(s,t)dsdt,
 \end{equation}
 where $\rho$ is a symmetric and continuous map from $U_T$ into $\mathbb R_+$ such that
 \begin{displaymath}
 \mathfrak c_R =
 \sup_{t\in [0,T]}\int_{0}^{T}\rho(s,t)ds <\infty.
 \end{displaymath}
 Then, for every $\varphi,\psi\in C^1([0,T];\mathbb R)$,
 \begin{eqnarray*}
  |\langle\varphi,\psi\rangle_R| & = &
  \left|\int_{[0,T]^2}\varphi(s)\rho(s,t)^{\frac{1}{2}}\psi(t)\rho(s,t)^{\frac{1}{2}}dsdt\right|\\
  & \leqslant &
  \left|\int_{0}^{T}\varphi(s)^2\int_{0}^{T}\rho(s,t)dtds\right|^{\frac{1}{2}}
  \left|\int_{0}^{T}\psi(t)^2\int_{0}^{T}\rho(s,t)dsdt\right|^{\frac{1}{2}}\leqslant
  \mathfrak c_R\|\varphi\|\cdot\|\psi\|.
 \end{eqnarray*}
 Here are two explicit examples of processes whose covariance function satisfies the condition (\ref{example_noise_covariance_condition_Gaussian_1_1}):
 \begin{enumerate}
  \item Let $R_H$ be the covariance function of the fractional Brownian motion of Hurst parameter $H\in (1/2,1)$, and note that
  \begin{displaymath}
  dR_H(s,t) =\rho_H(s,t)dsdt
  \quad\textrm{with}\quad
  \rho_H(s,t) =
  \underbrace{H(2H - 1)}_{=:\alpha_H}|t - s|^{2H - 2}.
  \end{displaymath}
  Since $H\in (1/2,1)$, for every $t\in [0,T]$,
  \begin{eqnarray*}
   \int_{0}^{T}\rho_H(s,t)ds & = &
   \alpha_H\left(\int_{0}^{t}s^{2H - 2}ds +\int_{0}^{T - t}s^{2H - 2}ds\right)\\
   & = &
   \frac{\alpha_H}{2H - 1}(t^{2H - 1} + (T - t)^{2H - 1})
   \leqslant
   2HT^{2H - 1},
  \end{eqnarray*}
  and then $R_H$ satisfies the condition (\ref{example_noise_covariance_condition_Gaussian_1_1}).
  \item Assume that
  \begin{displaymath}
  Z_t =\phi t +\sigma B_t
  \textrm{ $;$ }\forall t\in [0,T],
  \end{displaymath}
  where $\sigma\in\mathbb R^*$, $B$ is a fractional Brownian motion of Hurst parameter $H\in (1/2,1)$, and $\phi$ is a centered Gaussian random variable of variance $\sigma_{\phi}^{2} > 0$ such that $\phi$ and $B$ are independent. So,
  \begin{eqnarray*}
   dR(s,t) & = &
   \sigma_{\phi}^{2}dsdt +\sigma^2dR_H(s,t)\\
   & = &
   (\sigma_{\phi}^{2} +\sigma^2\alpha_H|t - s|^{2H - 2})dsdt.
  \end{eqnarray*}
  Since $H\in (1/2,1)$, $R$ satisfies the condition (\ref{example_noise_covariance_condition_Gaussian_1_1}).
 \end{enumerate}
\end{enumerate}
\end{example}
\noindent
Let us establish a risk bound on $\widehat b_m'$ - sharper than in Proposition \ref{risk_bound_estimator_b_0'_general} - by using Proposition \ref{Young_Wiener_isometry}.
%


%
\begin{proposition}\label{risk_bound_estimator_b_0'_Gaussian}
Under the conditions (\ref{noise_covariance_condition_Gaussian_1}) and (\ref{noise_covariance_condition_Gaussian_2}),
\begin{equation}\label{risk_bound_estimator_b_0'_Gaussian_1}
\mathbb E(\|\widehat b_m' - b_0'\|^2)\leqslant
\min_{\tau\in\mathcal S_m}\|\tau - b_0'\|^2 +
\frac{\mathfrak c_R}{\gamma_R}m\mathfrak R_N.
\end{equation}
\end{proposition}
%


%
\begin{proof}
First, by Equality (\ref{risk_bound_estimator_b_0'_general_3}),
\begin{equation}\label{risk_bound_estimator_b_0'_Gaussian_2}
\|\widehat b_m' - b_0'\|^2 =
\min_{\tau\in\mathcal S_m}\|\tau - b_0'\|^2 +
\sum_{j = 1}^{m}\left(\int_{0}^{T}\varphi_j(s)d\overline Z_s\right)^2.
\end{equation}
Now, by the condition (\ref{noise_covariance_condition_Gaussian_2}) on $Z^1,\dots,Z^N$, and by Proposition \ref{Young_Wiener_isometry}, for every $j\in\{1,\dots,m\}$,
\begin{eqnarray}
 \mathbb E\left[\left(\int_{0}^{T}\varphi_j(s)d\overline Z_s\right)^2\right]
 & = &
 \frac{1}{N^2}\sum_{i,k = 1}^{N}\mathbb E\left(
 \left(\int_{0}^{T}\varphi_j(s)dZ_{s}^{i}\right)
 \left(\int_{0}^{T}\varphi_j(s)dZ_{s}^{k}\right)\right)
 \nonumber\\
 \label{risk_bound_estimator_b_0'_Gaussian_3}
 & = &
 \frac{1}{N^2}\left(
 \int_{[0,T]^2}\varphi_j(s)\varphi_j(t)dR(s,t)\right)
 \sum_{i,k = 1}^{N}\Gamma_{\star}^{i,k} =
 \frac{\|\varphi_j\|_{R}^{2}}{N^2}\sum_{i,k = 1}^{N}\Gamma_{\star}^{i,k}.
\end{eqnarray}
Moreover, (\ref{noise_covariance_condition_Gaussian_1}) and (\ref{matrix_Gamma_Gaussian}) - respectively - lead to
\begin{displaymath}
\sum_{j = 1}^{m}\|\varphi_j\|_{R}^{2}
\leqslant
\mathfrak c_Rm
\quad {\rm and}\quad
\left|\sum_{i,k = 1}^{N}\Gamma_{\star}^{i,k}\right|
\leqslant
\frac{1}{\gamma_R}
\sum_{i,k = 1}^{N}\Gamma^{i,k}.
\end{displaymath}
Therefore, by Equality (\ref{risk_bound_estimator_b_0'_Gaussian_2}),
\begin{displaymath}
\mathbb E(\|\widehat b_m' - b_0'\|^2)
\leqslant
\min_{\tau\in\mathcal S_m}\|\tau - b_0'\|^2 +
\frac{\mathfrak c_R}{\gamma_R}\cdot
\frac{m}{N^2}
\sum_{i,k = 1}^{N}\Gamma^{i,k}.
\end{displaymath}
\end{proof}
%


%
\begin{remark}\label{remark_risk_bound_estimator_b_0'_Gaussian}
Assume that $(\varphi_1,\dots,\varphi_m)$ is the $[0,T]$-supported trigonometric basis, and that $b_0$ belongs to the Fourier-Sobolev space $\mathbb W_{2}^{\beta}([0,T])$. As mentioned in Remark \ref{remark_risk_bound_estimator_b_0'_general}.(1),
\begin{displaymath}
\min_{\tau\in\mathcal S_m}\|\tau - b_0'\|^2\leqslant\mathfrak c_{\beta}m^{-2\beta}.
\end{displaymath}
So, by Proposition \ref{risk_bound_estimator_b_0'_Gaussian},
\begin{displaymath}
\mathbb E(\|\widehat b_m' - b_0'\|^2)\lesssim
m^{-2\beta} + m\mathfrak R_N,
\end{displaymath}
and then $\widehat b_m'$ reaches the bias-variance tradeoff for
\begin{displaymath}
m\asymp
\mathfrak R_{N}^{-\frac{1}{2\beta + 1}},
\quad\textrm{leading to the rate $\mathfrak R_{N}^{\frac{2\beta}{2\beta + 1}}$ instead of $\mathfrak R_{N}^{\frac{2\beta}{2\beta + 3}}$},
\end{displaymath}
which illustrates that Proposition \ref{risk_bound_estimator_b_0'_Gaussian} provides a sharper risk bound on $\widehat b_m'$ than Proposition \ref{risk_bound_estimator_b_0'_general}.
\end{remark}
\noindent
Note that $\widehat b_m'$ is the minimizer, in $\mathcal S_m$, of the objective function $\gamma_N$ such that
\begin{displaymath}
\gamma_N(\tau) =
\|\tau\|^2 -\frac{2}{N}\sum_{i = 1}^{N}\int_{0}^{T}\tau(s)dX_{s}^{i}
\textrm{ $;$ }\forall\tau\in\mathcal S_m.
\end{displaymath}
This legitimates to consider the adaptive estimator $\widehat b' =\widehat b_{\widehat m}'$ with $\widehat m$ selected in $\mathcal M_N =\{1,\dots,N\}$ in the following way:
\begin{equation}\label{model_selection_criterion}
\widehat m =
\underset{m\in\mathcal M_N}{\rm argmin}
\{\gamma_N(\widehat b_m') + {\rm pen}(m)\},
\end{equation}
where
\begin{displaymath}
{\rm pen}(m) =
\mathfrak c_{\rm cal}m\mathfrak R_N
\textrm{ $;$ }\forall m\in\mathcal M_N,
\end{displaymath}
and $\mathfrak c_{\rm cal}\geqslant 4\delta^2$ ($\delta$ is defined by Equality (\ref{risk_bound_adaptive_estimator_b_0'_Gaussian_3}) in the proof of Theorem \ref{risk_bound_adaptive_estimator_b_0'_Gaussian}).
\\
\\
Under the conditions of Proposition \ref{risk_bound_estimator_b_0'_Gaussian}, the following theorem provides a risk bound on $\widehat b'$ when $\mathcal S_1,\dots,\mathcal S_N$ are nested models (i.e. $\mathcal S_m\subset\mathcal S_{m^{\star}}$ for every $m,m^{\star}\in\{1,\dots,N\}$ such that $m^{\star} > m$).
%


%
\begin{theorem}\label{risk_bound_adaptive_estimator_b_0'_Gaussian}
Assume that $\mathcal S_1,\dots,\mathcal S_N$ are nested models. Under the conditions (\ref{noise_covariance_condition_Gaussian_1}) and (\ref{noise_covariance_condition_Gaussian_2}),
\begin{displaymath}
\mathbb E(\|\widehat b' - b_0'\|^2)\lesssim
\min_{m\in\mathcal M_N}\left\{
\min_{\tau\in\mathcal S_m}\|\tau - b_0'\|^2 + m\mathfrak R_N\right\} +
\mathfrak R_N.
\end{displaymath}
\end{theorem}
%


%
\begin{proof}
The proof of Theorem \ref{risk_bound_adaptive_estimator_b_0'_Gaussian} is dissected in four steps.
\\
\\
{\bf Step 1.} First, for any $m\in\mathcal M_N$ and every $\tau,\tau^{\star}\in\mathcal S_{m\vee\widehat m}$,
\begin{eqnarray*}
 \gamma_N(\tau) -\gamma_N(\tau^{\star}) & = &
 \|\tau - b_0' + b_0'\|^2 - 2\mathcal I(\tau,\overline X) -
 (\|\tau^{\star} - b_0' + b_0'\|^2 - 2\mathcal I(\tau^{\star},\overline X))\\
 & = &
 \|\tau - b_0'\|^2 -\|\tau^{\star} - b_0'\|^2 +
 2\langle\tau -\tau^{\star},b_0'\rangle - 2\mathcal I(\tau -\tau^{\star},\overline X)\\
 & = &
 \|\tau - b_0'\|^2 -\|\tau^{\star} - b_0'\|^2 - 2\nu_N(\tau -\tau^{\star}),
\end{eqnarray*}
where
\begin{displaymath}
\nu_N(\varphi) =\mathcal I(\varphi,\overline Z)
\textrm{ $;$ }\forall\varphi\in\mathcal S_{m\vee\widehat m}.
\end{displaymath}
Then, by the definition of $\widehat m$,
\begin{eqnarray*}
 \|\widehat b' - b_0'\|^2
 & \leqslant &
 \|\widehat b_m' - b_0'\|^2 +
 2\nu_N(\widehat b' -\widehat b_m') + {\rm pen}(m) - {\rm pen}(\widehat m)\\
 & \leqslant &
 \|\widehat b_m' - b_0'\|^2 +\frac{1}{4}\|\widehat b' -\widehat b_m'\|^2 +
 4\mathcal E(m,\widehat m) + p(m,\widehat m) + {\rm pen}(m) - {\rm pen}(\widehat m)
\end{eqnarray*}
where, for every $m^{\star}\in\mathcal M_N$,
\begin{displaymath}
\mathcal E(m,m^{\star}) =
\left(\left[
\sup_{\tau\in\mathcal S_{m\vee m^{\star}} :\|\tau\| = 1}|\nu_N(\tau)|\right]^2 - p(m,m^{\star})\right)_+
\quad {\rm and}\quad
p(m,m^{\star}) =
\frac{\mathfrak c_{\rm cal}}{4}\cdot\frac{m\vee m^{\star}}{N^2}\sum_{i,k = 1}^{N}\Gamma^{i,k}.
\end{displaymath}
By using that $(a + b)^2\leqslant 2a^2 + 2b^2$ for every $a,b\in\mathbb R$, and since $4p(m,\widehat m)\leqslant {\rm pen}(m) + {\rm pen}(\widehat m)$,
\begin{equation}\label{risk_bound_adaptive_estimator_b_0'_Gaussian_1}
\|\widehat b' - b_0'\|^2
\leqslant
3\|\widehat b_m' - b_0'\|^2 + 8\mathcal E(m,\widehat m) + 4{\rm pen}(m).
\end{equation}
{\bf Step 2.} Consider $\tau\in\mathcal S_N$, and note that
\begin{displaymath}
\nu_N(\tau) =
\frac{1}{N}\sum_{i = 1}^{N}\int_{0}^{T}\tau(s)dZ_{s}^{i}
\rightsquigarrow\mathcal N(0,\sigma_{N}^{2}(\tau)),
\end{displaymath}
where
\begin{eqnarray*}
 \sigma_{N}^{2}(\tau) & = &
 \frac{1}{N^2}\sum_{i,k = 1}^{N}
 \mathbb E\left(\left(\int_{0}^{T}\tau(s)dZ_{s}^{i}\right)
 \left(\int_{0}^{T}\tau(s)dZ_{s}^{k}\right)\right)\\
 & = &
 \frac{1}{N^2}\left(\int_{[0,T]^2}\tau(s)\tau(t)dR(s,t)\right)
 \sum_{i,k = 1}^{N}\Gamma_{\star}^{i,k} =
 \frac{\|\tau\|_{R}^{2}}{N^2}\sum_{i,k = 1}^{N}\Gamma_{\star}^{i,k}
 \quad\textrm{by Proposition \ref{Young_Wiener_isometry}.}
\end{eqnarray*}
For any $\varepsilon > 0$,
\begin{eqnarray*}
 \mathbb P(|\nu_N(\tau)| >\varepsilon)
 & = &
 2\mathbb P(\sigma_N(\tau)\xi >\varepsilon)
 \quad {\rm with}\quad
 \xi\rightsquigarrow\mathcal N(0,1)\\
 & = &
 \frac{2}{\sqrt{2\pi}}
 \int_{\frac{\varepsilon}{\sigma_N(\tau)}}^{\infty}
 \exp\left(-\frac{x^2}{2}\right)dx
 \leqslant
 \exp\left(
 -\frac{1}{4}\cdot\frac{\varepsilon^2}{\sigma_{N}^{2}(\tau)}\right)
 \underbrace{
 \frac{2}{\sqrt{2\pi}}\int_{-\infty}^{\infty}\exp\left(-\frac{x^2}{4}\right)dx}_{=:\mathfrak c_1}.
\end{eqnarray*}
Moreover, (\ref{noise_covariance_condition_Gaussian_1}) and (\ref{matrix_Gamma_Gaussian}) - respectively - lead to
\begin{displaymath}
\|\tau\|_{R}^{2}
\leqslant
\mathfrak c_R\|\tau\|^2
\quad {\rm and}\quad
\left|\sum_{i,k = 1}^{N}\Gamma_{\star}^{i,k}\right|
\leqslant
\frac{1}{\gamma_R}
\sum_{i,k = 1}^{N}\Gamma^{i,k}.
\end{displaymath}
Then,
\begin{equation}\label{risk_bound_adaptive_estimator_b_0'_Gaussian_2}
\mathbb P(|\nu_N(\tau)| >\varepsilon)
\leqslant
\mathfrak c_1
\exp\left(-\mathfrak c_2
\frac{N^2\varepsilon^2}{\|\tau\|^2
\sum_{i,k = 1}^{N}\Gamma^{i,k}}\right)\\
\quad {\rm with}\quad
\mathfrak c_2 =\frac{\gamma_R}{4\mathfrak c_R}.
\end{equation}
{\bf Step 3.} By using Inequality (\ref{risk_bound_adaptive_estimator_b_0'_Gaussian_2}), and by following the pattern of the proof of Baraud et al. \cite{BCV01}, Proposition 6.1, the purpose of this step is to find a suitable bound on $\mathbb E(\mathcal E(m,m^{\star}))$ ($m,m^{\star}\in\mathcal M_N$). Consider $\delta_0\in (0,1)$ and let $(\delta_n)_{n\in\mathbb N^*}$ be the real sequence defined by
\begin{displaymath}
\delta_n :=\delta_02^{-n}
\textrm{ $;$ }
\forall n\in\mathbb N^*.
\end{displaymath}
Since $\mathcal S_{m\vee m^{\star}}$ is a vector subspace of $\mathbb L^2([0,T],dt)$ of dimension $m\vee m^{\star}$, by Lorentz et al. \cite{LGM96}, Chapter 15, Proposition 1.3, for any $n\in\mathbb N$, there exists $T_n\subset\mathcal B_{m,m^{\star}} =\{\tau\in\mathcal S_{m\vee m^{\star}} :\|\tau\| = 1\}$ such that $|T_n|\leqslant (3/\delta_n)^{m\vee m^{\star}}$ and, for any $\tau\in\mathcal B_{m,m^{\star}}$,
\begin{displaymath}
\exists f_n\in T_n :
\|\tau - f_n\|\leqslant\delta_n.
\end{displaymath}
In particular, note that
\begin{displaymath}
\tau =
f_0 +\sum_{n = 1}^{\infty}(f_n - f_{n - 1}).
\end{displaymath}
Then, for any sequence $(\Delta_n)_{n\in\mathbb N}$ of elements of $(0,\infty)$ such that $\Delta :=\sum_{n\in\mathbb N}\Delta_n <\infty$,
\begin{eqnarray*}
 & & \mathbb P\left(\left[
 \sup_{\tau\in\mathcal B_{m,m^{\star}}}
 |\nu_N(\tau)|\right]^2 >
 \Delta^2\right)\\
 & &
 \hspace{2cm} =
 \mathbb P\left(\exists (f_n)_{n\in\mathbb N}\in\prod_{n = 0}^{\infty}T_n :
 |\nu_N(f_0)| +\sum_{n = 1}^{\infty}
 |\nu_N(f_n - f_{n - 1})| >\Delta\right)\\
 & &
 \hspace{2cm}\leqslant
 \mathbb P\left(\exists (f_n)_{n\in\mathbb N}\in\prod_{n = 0}^{\infty}T_n :
 |\nu_N(f_0)| >\Delta_0\textrm{ or }[\exists n\in\mathbb N^* :
 |\nu_N(f_n - f_{n - 1})| >\Delta_n]\right)\\
 & &
 \hspace{2cm}\leqslant
 \sum_{f_0\in T_0}\mathbb P(|\nu_N(f_0)| >\Delta_0) +
 \sum_{n = 1}^{\infty}\sum_{(f_{n - 1},f_n)\in\mathbb T_n}
 \mathbb P(|\nu_N(f_n - f_{n - 1})| >\Delta_n)
\end{eqnarray*}
with $\mathbb T_n = T_{n - 1}\times T_n$ for every $n\in\mathbb N^*$. Moreover, $\|f_0\|^2\leqslant 1$ and, for every $n\in\mathbb N^*$,
\begin{displaymath}
\|f_n - f_{n - 1}\|^2
\leqslant
2\delta_{n - 1}^{2} + 2\delta_{n}^{2} =\frac{5}{2}\delta_{n - 1}^{2}.
\end{displaymath}
So, by Inequality (\ref{risk_bound_adaptive_estimator_b_0'_Gaussian_2}),
\begin{eqnarray}
 \mathbb P\left(\left[
 \sup_{\tau\in\mathcal B_{m,m^{\star}}}
 |\nu_N(\varphi)|\right]^2 >\Delta^2\right)
 & \leqslant &
 \mathfrak c_1\sum_{f_0\in T_0}\exp\left(
 -\frac{\mathfrak c_2\Delta_{0}^{2}}{\|f_0\|^2\mathfrak R_N}\right)
 \nonumber\\
 & &
 \hspace{2cm} +
 \mathfrak c_1\sum_{n = 1}^{\infty}\sum_{(f_{n - 1},f_n)\in\mathbb T_n}
 \exp\left(
 -\frac{\mathfrak c_2\Delta_{n}^{2}}{\|f_n - f_{n - 1}\|^2\mathfrak R_N}\right)
 \nonumber\\
 \label{risk_bound_adaptive_estimator_b_0'_Gaussian_4}
 & \leqslant &
 \mathfrak c_1\exp\left(h_0 -\frac{\mathfrak c_2\Delta_{0}^{2}}{
 \mathfrak R_N}\right) +
 \mathfrak c_1\sum_{n = 1}^{\infty}\exp\left(h_{n - 1} + h_n -
 \frac{\mathfrak c_3\Delta_{n}^{2}}{\delta_{n - 1}^{2}
 \mathfrak R_N}\right)
\end{eqnarray}
with $\mathfrak c_3 = 2\mathfrak c_2/5$, and $h_n =\log(|T_n|)$ for every $n\in\mathbb N$. Now, let us take $\Delta_0$ such that
\begin{displaymath}
h_0 -\frac{\mathfrak c_2\Delta_{0}^{2}}{\mathfrak R_N} = -(m\vee m^{\star} + x)
\quad {\rm with}\quad x > 0,
\end{displaymath}
which leads to
\begin{displaymath}
\Delta_0 =
\left(\frac{\mathfrak R_N}{\mathfrak c_2}(m\vee m^{\star} + x + h_0)\right)^{\frac{1}{2}}\\
\leqslant
\left(\frac{\mathfrak R_N}{\mathfrak c_2}(1 + h_0)(m\vee m^{\star} + x)\right)^{\frac{1}{2}},
\end{displaymath}
and for every $n\in\mathbb N^*$, let us take $\Delta_n$ such that
\begin{displaymath}
h_{n - 1} + h_n -
\frac{\mathfrak c_3\Delta_{n}^{2}}{\delta_{n - 1}^{2}
\mathfrak R_N} =
-(m\vee m^{\star} + x + n),
\end{displaymath}
which leads to
\begin{eqnarray*}
 \Delta_n & = &
 \left(\frac{\delta_{n - 1}^{2}\mathfrak R_N}{\mathfrak c_3}
 (m\vee m^{\star} + x + h_{n - 1} + h_n + n)\right)^{\frac{1}{2}}\\
 & \leqslant &
 \left(\frac{\delta_{n - 1}^{2}\mathfrak R_N}{\mathfrak c_3}
 (1 +\lambda_n + n)(m\vee m^{\star} + x)\right)^{\frac{1}{2}}
 \quad {\rm with}\quad
 \lambda_n =
 2\log\left(\frac{3}{\delta_0}\right) + (2n - 1)\log(2).
\end{eqnarray*}
For this appropriate sequence $(\Delta_n)_{n\in\mathbb N}$,
\begin{displaymath}
\mathbb P\left(\left[
\sup_{\tau\in\mathcal B_{m,m^{\star}}}
|\nu_N(\tau)|\right]^2 >\Delta^2\right)
\leqslant
\mathfrak c_1e^{-x}e^{-(m\vee m^{\star})}\left(1 +\sum_{n = 1}^{\infty}e^{-n}\right)
\leqslant
1.6\mathfrak c_1e^{-x}e^{-(m\vee m^{\star})}
\end{displaymath}
by Inequality \eqref{risk_bound_adaptive_estimator_b_0'_Gaussian_4}, and
\begin{eqnarray*}
 \Delta^2 & \leqslant &
 \mathfrak R_N
 \left(\mathfrak c_{2}^{-\frac{1}{2}}(1 + h_0)^{\frac{1}{2}}(m\vee m^{\star} + x)^{\frac{1}{2}} +
 \mathfrak c_{3}^{-\frac{1}{2}}\sum_{n = 1}^{\infty}
 \delta_{n - 1}(1 +\lambda_n + n)^{\frac{1}{2}}(m\vee m^{\star} + x)^{\frac{1}{2}}\right)^2\\
 & \leqslant &
 \mathfrak R_N(m\vee m^{\star} + x)\delta^2
\end{eqnarray*}
with
\begin{equation}
\label{risk_bound_adaptive_estimator_b_0'_Gaussian_3}
\delta =
\mathfrak c_{2}^{-\frac{1}{2}}(1 + h_0)^{\frac{1}{2}} +
\mathfrak c_{3}^{-\frac{1}{2}}
\sum_{n = 1}^{\infty}\delta_{n - 1}(1 +\lambda_n + n)^{\frac{1}{2}} <\infty.
\end{equation}
So,
\begin{displaymath}
\mathbb P\left(\left[
\sup_{\tau\in\mathcal B_{m,m^{\star}}}
|\nu_N(\tau)|\right]^2 -
\frac{4\delta^2}{\mathfrak c_{\rm cal}}p(m,m^{\star}) >
\mathfrak R_N\delta^2x\right)
\leqslant
1.6\mathfrak c_1e^{-x}e^{-(m\vee m^{\star})}
\end{displaymath}
and then, by taking $\mathfrak c_{\rm cal} > 4\delta^2$ and $y =\mathfrak R_N\delta^2x$,
\begin{displaymath}
\mathbb P\left(\left[
\sup_{\tau\in\mathcal B_{m,m^{\star}}}
|\nu_N(\tau)|\right]^2 - p(m,m^{\star}) > y\right)
\leqslant
1.6\mathfrak c_1\exp\left(-\frac{y}{\mathfrak R_N\delta^2}\right)e^{-(m\vee m^{\star})}.
\end{displaymath}
Therefore,
\begin{eqnarray}
 \mathbb E(\mathcal E(m,m^{\star}))
 & = &
 \int_{0}^{\infty}\mathbb P\left(\left[
 \sup_{\tau\in\mathcal B_{m,m^{\star}}}
 |\nu_N(\tau)|\right]^2 - p(m,m^{\star}) > y\right)dy
 \nonumber\\
 \label{risk_bound_adaptive_estimator_b_0'_Gaussian_5}
 & \leqslant &
 \mathfrak c_4
 \frac{e^{-(m\vee m^{\star})}}{N^2}\sum_{i,k = 1}^{N}\Gamma^{i,k}
 \quad {\rm with}\quad
 \mathfrak c_4 = 1.6\mathfrak c_1\delta^2.
\end{eqnarray}
{\bf Step 4 (conclusion).} By Inequality (\ref{risk_bound_adaptive_estimator_b_0'_Gaussian_5}), there exists a constant $\mathfrak c_5 > 0$, not depending on $N$, such that for every $m\in\mathcal M_N$,
\begin{eqnarray*}
 \mathbb E(\mathcal E(m,\widehat m))
 & \leqslant &
 \sum_{m^{\star}\in\mathcal M_N}\mathbb E(\mathcal E(m,m^{\star}))\\
 & \leqslant &
 \frac{\mathfrak c_4}{N^2}\left(\sum_{m^{\star} = 1}^{N}e^{-(m\vee m^{\star})}\right)
 \sum_{i,k = 1}^{N}\Gamma^{i,k}
 \leqslant
 \frac{\mathfrak c_5}{N^2}\sum_{i,k = 1}^{N}\Gamma^{i,k}.
\end{eqnarray*}
So, by Inequality (\ref{risk_bound_adaptive_estimator_b_0'_Gaussian_1}),
\begin{displaymath}
\mathbb E(\|\widehat b' - b_0'\|^2)
\leqslant
\min_{m\in\mathcal M_N}\{3\mathbb E(\|\widehat b_m' - b_0'\|^2) + 4{\rm pen}(m)\} +
\frac{8\mathfrak c_5}{N^2}\sum_{i,k = 1}^{N}\Gamma^{i,k},
\end{displaymath}
and then Proposition \ref{risk_bound_estimator_b_0'_Gaussian} allows to conclude.
\end{proof}
%


%
\subsection{Back to the models of Section \ref{section_examples}}\label{section_estimation_examples}
First, let us apply Proposition \ref{risk_bound_estimator_b_0} in Models (\ref{linear_SDE_correlated_noises_transformed}), (\ref{linear_fDE_random_effect_transformed}), (\ref{interacting_particle_system_transformed}) and (\ref{long_time_observation_copies_fractional_model}).
%


%
\begin{corollary}\label{risk_bound_estimator_b_0_linear_fractional_diffusions}
If $X^1,\dots,X^N$ are defined by (\ref{linear_SDE_correlated_noises_transformed}), then
\begin{displaymath}
\mathbb E(\|\widehat b - b_0\|^2)
\leqslant
\frac{(\sigma T)^2}{N^2}\sum_{i,k = 1}^{N}|R^{i,k}|,
\end{displaymath}
and if $X^1,\dots,X^N$ are defined by (\ref{linear_fDE_random_effect_transformed}), then
\begin{displaymath}
\mathbb E(\|\widehat b - b_0\|^2)
\leqslant\frac{\Sigma^2}{N}
\quad\textrm{with}\quad
\Sigma =\sqrt{T(\sigma_{\phi}^{2}T^2 +\sigma^2T^{2H})}.
\end{displaymath}
\end{corollary}
%


%
\begin{corollary}\label{risk_bound_estimator_b_0_linteracting_particle_system}
Assume that $X^1,\dots,X^N$ are defined by (\ref{interacting_particle_system_transformed}), and consider
\begin{displaymath}
\widehat{\tt b}(t) =
\widehat b(t) -
\int_{0}^{t}e^{-(t - s)}\widehat b(s)ds
\textrm{ $;$ }t\in [0,T].
\end{displaymath}
Then,
\begin{displaymath}
\mathbb E(\|\widehat b - b_0\|^2)\leqslant
\frac{4T(1\vee T)^3\sigma^2}{N}
\quad\textrm{and}\quad
\mathbb E(\|\widehat{\tt b} - {\tt b}_0\|^2)\leqslant
\frac{8T(1 + T^2)(1\vee T)^3\sigma^2}{N}.
\end{displaymath}
\end{corollary}
%


%
\begin{proof}
First, by Proposition \ref{matrix_Gamma_interacting_particle_system},
\begin{displaymath}
\sum_{i,k = 1}^{N}\overline\Gamma^{i,k} =
T(1\vee T)^3\sigma^2\sum_{i,k = 1}^{N}\left(\mathbf 1_{i = k} +\frac{3}{N}\right) =
4T(1\vee T)^3\sigma^2N.
\end{displaymath}
Then, by Proposition \ref{risk_bound_estimator_b_0},
\begin{displaymath}
\mathbb E(\|\widehat b - b_0\|^2)
\leqslant
\frac{4T(1\vee T)^3\sigma^2}{N}.
\end{displaymath}
Now, since the derivative $b_0'$ of $b_0$ satisfies $b_0' =\texttt b_0' +\texttt b_0$ with $\texttt b_0(0) = 0$, for every $t\in [0,T]$,
\begin{displaymath}
\texttt b_0(t) =
\int_{0}^{t}e^{-(t - s)}b_0'(s)ds =
b_0(t) -\int_{0}^{t}e^{-(t - s)}b_0(s)ds,
\end{displaymath}
leading to
\begin{eqnarray*}
 |\widehat{\tt b}(t) -\texttt b_0(t)|^2
 & = &
 \left|\widehat b(t) - b_0(t) -
 \int_{0}^{t}e^{-(t - s)}(\widehat b(s) - b_0(s))ds\right|^2\\
 & \leqslant &
 2|\widehat b(t) - b_0(t)|^2 +
 2T\int_{0}^{T}|\widehat b(s) - b_0(s)|^2ds.
\end{eqnarray*}
So,
\begin{displaymath}
\mathbb E(\|\widehat{\tt b} -\texttt b_0\|^2)
\leqslant
2(1 + T^2)\mathbb E(\|\widehat b - b_0\|^2)
\leqslant
\frac{8T(1 + T^2)(1\vee T)^3\sigma^2}{N}.
\end{displaymath}
\end{proof}
%


%
\begin{corollary}\label{risk_bound_estimator_b_0_long_time_drifted_fBm}
If $X^1,\dots,X^N$ are defined by (\ref{long_time_observation_copies_fractional_model}) with $\Delta\geqslant T$, then
\begin{displaymath}
\mathbb E(\|\widehat b - b_0\|^2)
\lesssim
\frac{1}{N} +
(\Delta N)^{2H - 2}\mathbf 1_{H >\frac{1}{2}}.
\end{displaymath}
\end{corollary}
%


%
\begin{proof}
Consider $\mathfrak c_{\ref{matrix_Gamma_long_time_drifted_fBm}} =\max_{i = 1,2}\mathfrak c_{\ref{matrix_Gamma_long_time_drifted_fBm},i}(H)$. By Propositions \ref{matrix_Gamma_long_time_drifted_fBm} and \ref{risk_bound_estimator_b_0},
\begin{displaymath}
\mathbb E(\|\widehat b - b_0\|^2)
\leqslant
\frac{\mathfrak c_{\ref{matrix_Gamma_long_time_drifted_fBm}}}{N^2}
\left[N +\left(\Delta^{2H - 2}\sum_{i\neq k}|k - i|^{2H - 2}\right)\mathbf 1_{H\neq\frac{1}{2}}\right].
\end{displaymath}
If $H < 1/2$,
\begin{eqnarray*}
 \sum_{i\neq k}|k - i|^{2H - 2}
 & = &
 2\sum_{i = 1}^{N}\sum_{k = i + 1}^{N}(k - i)^{2H - 2}\\
 & \leqslant &
 2N\sum_{k = 1}^{N}k^{2H - 2}
 \leqslant
 \mathfrak c_1N
 \quad {\rm with}\quad
 \mathfrak c_1 =\sum_{n = 1}^{\infty}\frac{2}{n^{2 - 2H}} <\infty,
\end{eqnarray*}
and then
\begin{displaymath}
\mathbb E(\|\widehat b - b_0\|^2)
\leqslant
\frac{\mathfrak c_{\ref{matrix_Gamma_long_time_drifted_fBm}}}{N}
(1 +\mathfrak c_1T^{2H - 2}).
\end{displaymath}
If $H > 1/2$,
\begin{displaymath}
\sum_{i\neq k}|k - i|^{2H - 2}\leqslant
2N\sum_{k = 1}^{N}k^{2H - 2} =
2N^{2H}\cdot
\underbrace{\frac{1}{N}\sum_{k = 1}^{N}\left(\frac{k}{N}\right)^{2H - 2}}_{\sim (2H - 1)^{-1}},
\end{displaymath}
and then there exists a constant $\mathfrak c_2 > 0$, not depending on $m$, $N$ and $\Delta$, such that
\begin{displaymath}
\mathbb E(\|\widehat b - b_0\|^2)
\leqslant
\mathfrak c_{\ref{matrix_Gamma_long_time_drifted_fBm}}
\left(\frac{1}{N} +\mathfrak c_2(\Delta N)^{2H - 2}\right).
\end{displaymath}
This concludes the proof.
\end{proof}
%


%
\begin{remark}\label{remark_risk_bound_estimator_b_0_long_time_drifted_fBm}
Assume that $H > 1/2$ and $\Delta\geqslant T\vee N^{\frac{2H - 1}{2 - 2H}}$. By Corollary \ref{risk_bound_estimator_b_0_long_time_drifted_fBm},
\begin{displaymath}
\mathbb E(\|\widehat b - b_0\|^2)\lesssim\frac{1}{N}.
\end{displaymath}
\end{remark}
\noindent
Now, although Corollary \ref{risk_bound_estimator_b_0_linear_fractional_diffusions} provides suitable risk bounds on $\widehat b$ - the estimator of $b_0$ - in Models (\ref{linear_SDE_correlated_noises_transformed}) and (\ref{linear_fDE_random_effect_transformed}), the function of interest is that denoted by $\texttt b_0$. The conditions (\ref{noise_covariance_condition_Gaussian_2}) and (\ref{noise_covariance_condition_Gaussian_1}) (see Example \ref{example_noise_covariance_condition_Gaussian_1}) are satisfied in each of these models, and then one can deduce risk bounds on adaptive estimators of $\texttt b_0$ from Theorem \ref{risk_bound_adaptive_estimator_b_0'_Gaussian}. Since $b_0 =\int_{0}^{.}(\texttt b_0 -\frac{\sigma^2}{2})$ (resp. $b_0 =\int_{0}^{.}\texttt b_0$) in Model (\ref{linear_SDE_correlated_noises_transformed}) (resp. (\ref{linear_fDE_random_effect_transformed})), $\widehat{\tt b} =\widehat b' +\frac{\sigma^2}{2}$ (resp. $\widehat b'$) is an estimator of $\texttt b_0$.
%


%
\begin{corollary}\label{risk_bound_estimator_b_0'_linear_fractional_diffusions}
If $X^1,\dots,X^N$ are defined by (\ref{linear_SDE_correlated_noises_transformed}), then
\begin{displaymath}
\mathbb E(\|\widehat{\tt b} - {\tt b}_0\|^2)
\lesssim
\min_{m\in\mathcal M_N}\left\{
\min_{\tau\in\mathcal S_m}\|\tau - {\tt b}_0\|^2 + m\mathcal R_N\right\} +
\mathcal R_N
\quad\textrm{with}\quad
\mathcal R_N =
\frac{1}{N^2}\sum_{i,k = 1}^{N}|R^{i,k}|,
\end{displaymath}
and if $X^1,\dots,X^N$ are defined by (\ref{linear_fDE_random_effect_transformed}) with $H\in (1/2,1)$, then
\begin{displaymath}
\mathbb E(\|\widehat b' - {\tt b}_0\|^2)
\lesssim
\min_{m\in\mathcal M_N}\left\{
\min_{\tau\in\mathcal S_m}\|\tau - {\tt b}_0\|^2 +\frac{m}{N}\right\} +\frac{1}{N}.
\end{displaymath}
\end{corollary}
%


%
\section{Numerical experiments}\label{section_numerical_experiments}
%


%
\subsection{Estimation in Models (\ref{interacting_particle_system}) and (\ref{long_time_observation_copies_fractional_model}) via $\widehat b$}\label{subsection_numerical_experiments_b_0}
First, $N = 100$ independent Brownian motions $W^1,\dots,W^N$ are simulated along the dissection $\{\frac{\ell T}{n}\textrm{ $;$ }\ell = 0,\dots,n\}$ of $[0,T]$ with $T = 1$ and $n = 150$, and so are the processes $Y^1,\dots,Y^N$ defined by (\ref{interacting_particle_system}) thanks to the Euler scheme with $Y_0 = 5$, $\sigma = 0.5$ and
\begin{displaymath}
\texttt b_0(t) = t^2
\textrm{ $;$ }\forall t\in [0,T].
\end{displaymath}
For any $i\in\{1,\dots,N\}$, the transformation (\ref{interacting_particle_system_transformed}) is applied to $Y^i$, leading to $X^i = b_0 + Z^i$ where, for every $t\in [0,T]$,
\begin{displaymath}
b_0(t) = t^2\left(1 +\frac{t}{3}\right)
\quad {\rm and}\quad
\frac{1}{\sigma}Z_{t}^{i} =
W_{t}^{i} +\frac{1}{N}\sum_{\ell = 1}^{N}\int_{0}^{t}W_{s}^{\ell}ds.
\end{displaymath}
Then, $b_0$ is estimated by $\widehat b =\overline X$, and $\texttt b_0$ by
\begin{displaymath}
\widehat{\tt b}(t) =\widehat b(t) -\int_{0}^{t}e^{-(t - s)}\widehat b(s)ds
\textrm{ $;$ }t\in [0,T].
\end{displaymath}
Figure \ref{plots_IPS} shows the true function $\texttt b_0$ (red line) and 5 estimations obtained thanks to $\widehat{\tt b}$ (dashed black lines). The experiment is repeated $100$ times, and both the mean and standard deviation of the integrated squared error (ISE) of $\widehat{\tt b}$ are very small: $8.7\cdot 10^{-4}$ and $4.4\cdot 10^{-4}$ respectively.
\begin{figure}[h!]
\includegraphics[width=0.5\textwidth]{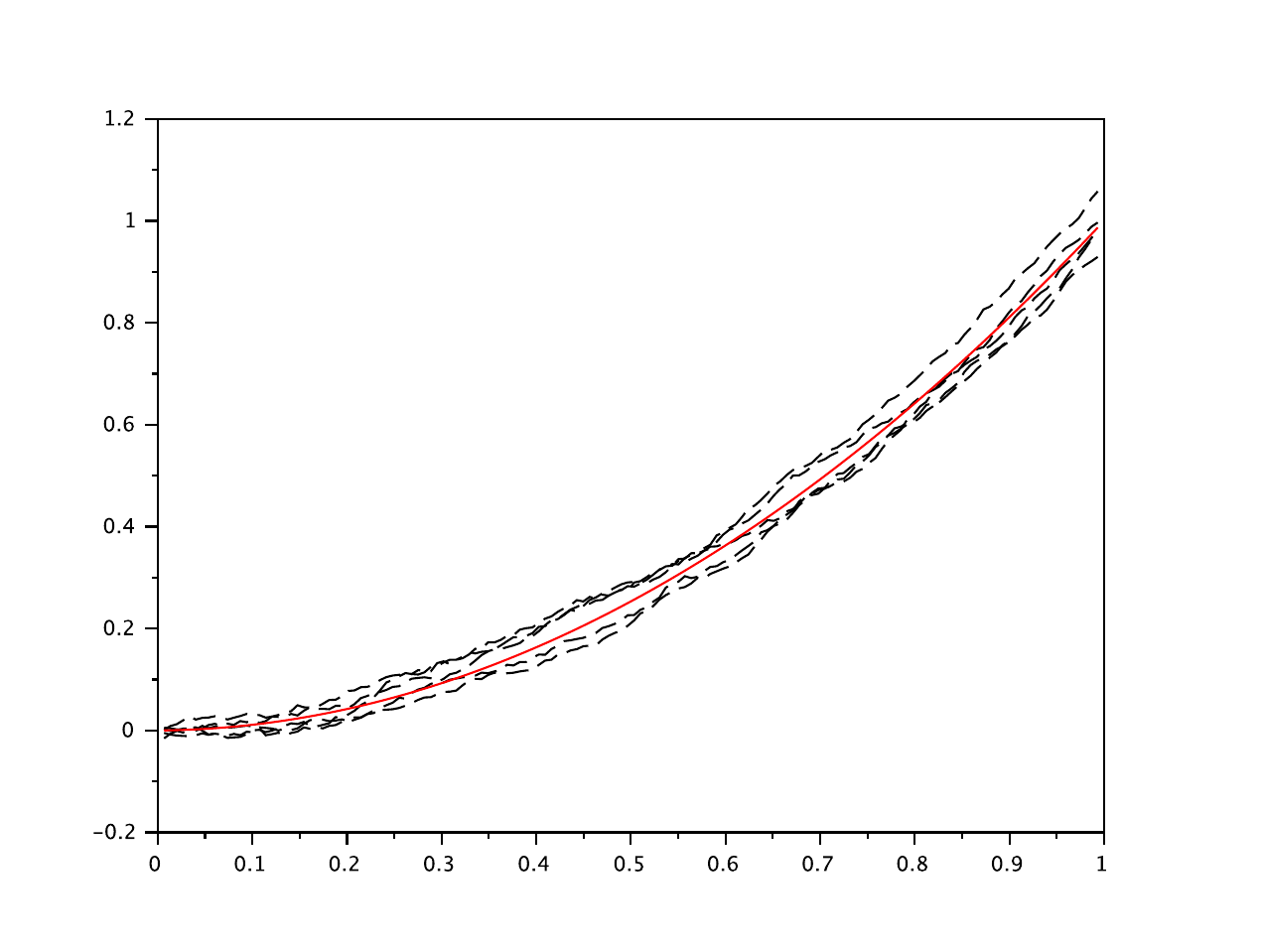}
\caption{Plots of $5$ estimations (dashed black) of the true function $\texttt b_0$ (red) in Model (\ref{interacting_particle_system}) ($N = 100$ copies).}
\label{plots_IPS}
\end{figure}
\newline
Now, consider $H\in (1/2,1)$, $T = 1$, $N = 50$, $\Delta =\delta T$ with $\delta\in\mathbb N^*$, and the $T$-periodic (and then $(T +\Delta)$-periodic) function $b_0$ such that
\begin{displaymath}
b_0(t) = t^2
\textrm{ $;$ }\forall t\in [0,T).
\end{displaymath}
Note that, for every $i\in\{1,\dots,N + 1\}$,
\begin{displaymath}
t_i(\Delta) = (i - 1)(T +\Delta) = (i - 1)(1 +\delta)T.
\end{displaymath}
For $n = 50$, the fractional Brownian motion $B$, and then $X = b_0 + Z$ with $Z =\sigma B$, are simulated along the dissection
\begin{displaymath}
\bigcup_{i = 1}^{N}
\left\{t_i(\Delta) +\frac{\ell T}{n}\textrm{ $;$ }\ell = 0,\dots,n(1 +\delta)\right\}
\quad {\rm of}\quad
[0,t_{N + 1}(\Delta)].
\end{displaymath}
So, for every $i\in\{1,\dots,N\}$, the process $X^i$ defined by (\ref{long_time_observation_copies_fractional_model}) is computed along $\{\frac{\ell T}{n}\textrm{ $;$ }\ell = 0,\dots,n\}$:
\begin{displaymath}
X_{\frac{\ell T}{n}}^{i} =
X_{t_i(\Delta) +\frac{\ell T}{n}} - X_{t_i(\Delta)}
\textrm{ $;$ }\ell\in\{0,\dots,n\}.
\end{displaymath}
Again, $b_0$ is estimated by $\widehat b =\overline X$. Figure \ref{plots_FBM} shows the true function $b_0$ (red line) and 5 estimations obtained thanks to $\widehat b$ (dashed black lines) for two values of $H$ ($0.6$ and $0.9$) and a double forgetting period ($\delta = 2$).
\begin{figure}[h!]
\begin{tabular}{cc}
 \includegraphics[width=0.4\textwidth]{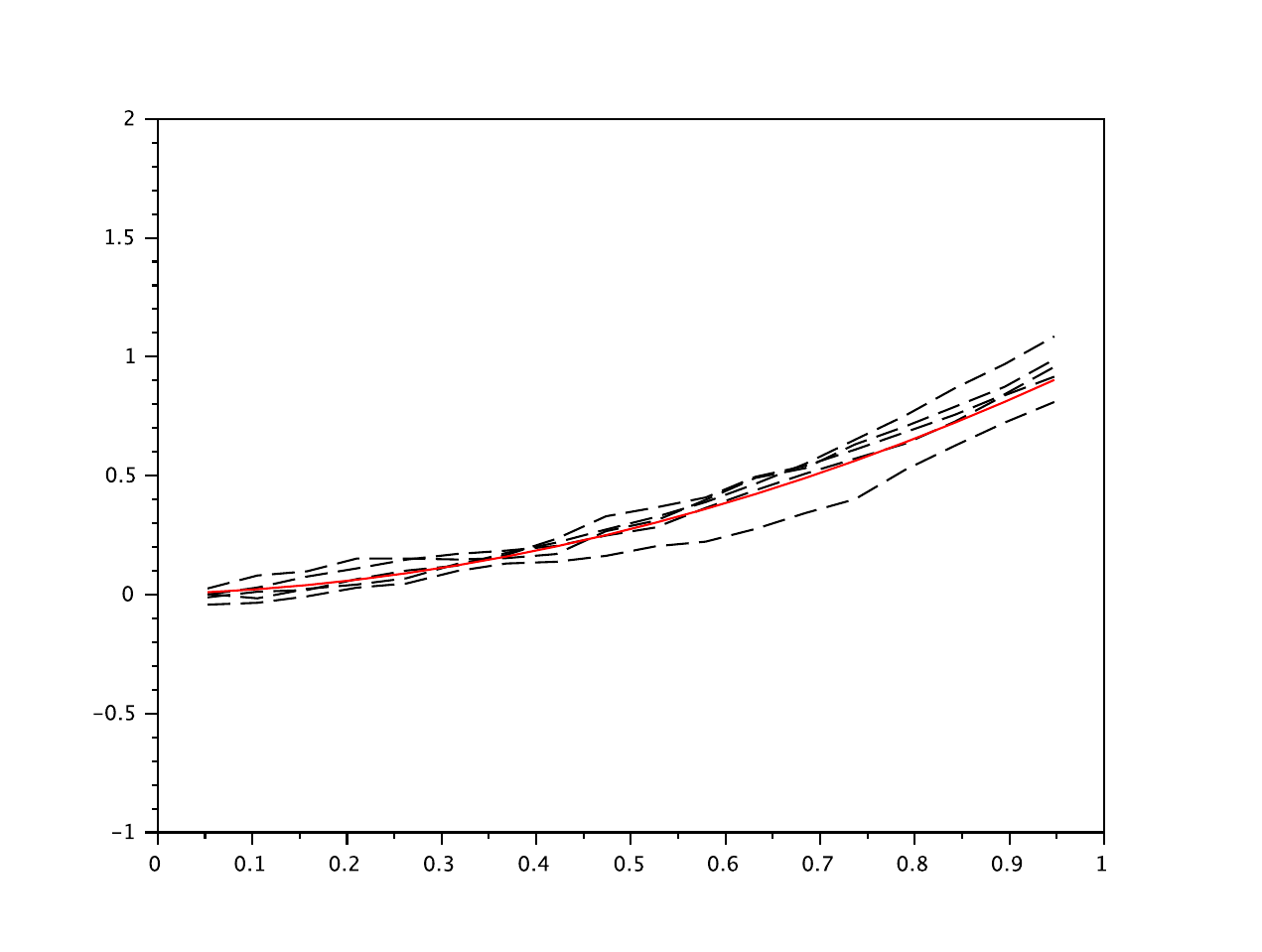} &
 \includegraphics[width=0.4\textwidth]{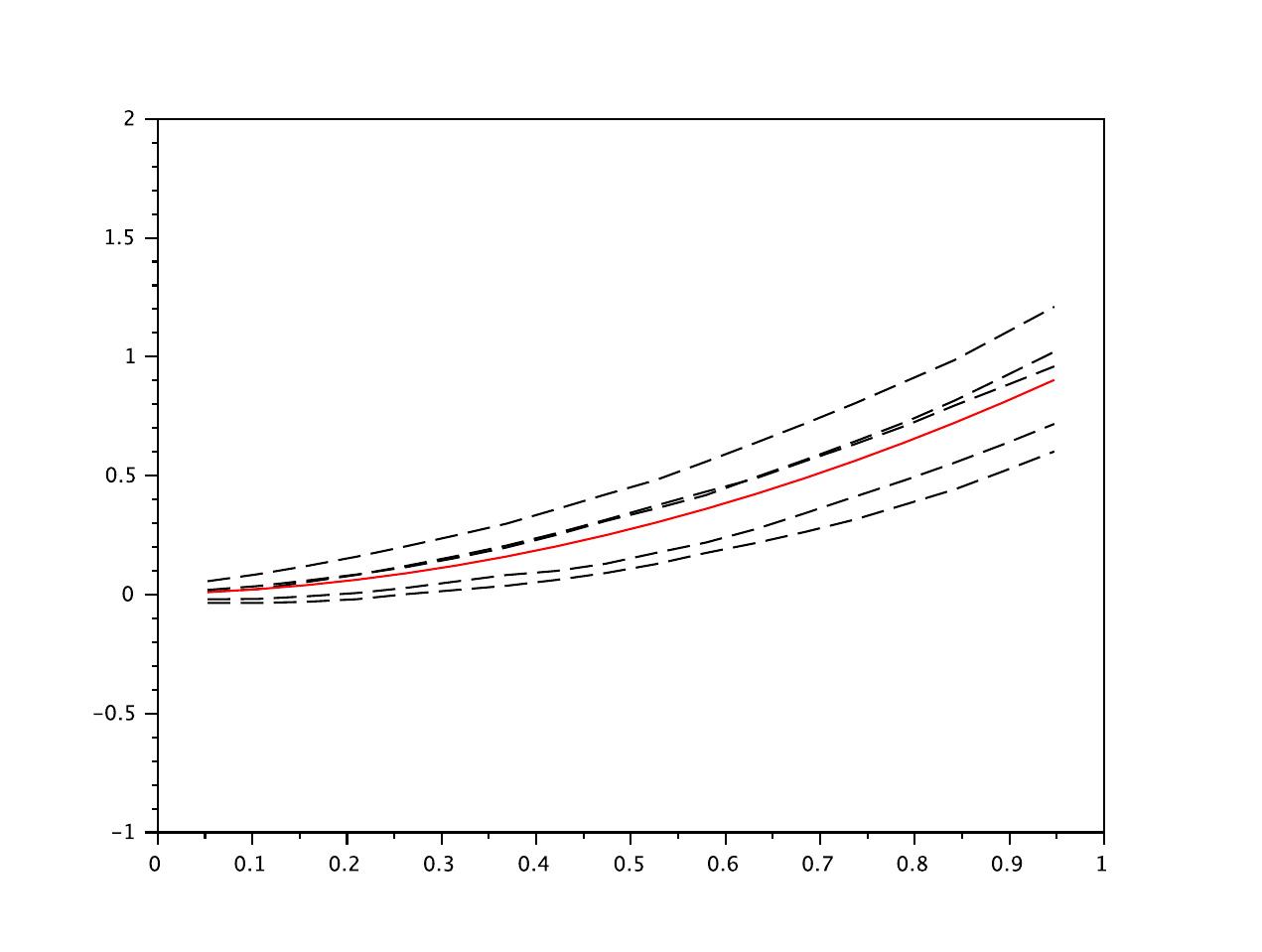}
\end{tabular}
\caption{Plots of 5 estimations (dashed black) of the true function $b_0$ (red) in Model (\ref{long_time_observation_copies_fractional_model}) for $\delta = 2$ and $H = 0.6,0.9$ ($N = 50$ copies).}
\label{plots_FBM} 
\end{figure}
The experiment is repeated $100$ times for $\delta = 2$, but also for a single forgetting period ($\delta = 1$), and the means and standard deviations of ${\rm ISE}(\widehat b)$ are stored in Table \ref{table_MISE_FBM}.
\begin{table}[h!]
\begin{center}
\begin{tabular}{|l||c|c|}
 \hline
  & $\delta = 1$ & $\delta = 2$\\
 \hline
 \hline
 $H = 0.6$ & $5\cdot 10^{-3}$ ($5.9\cdot 10^{-3}$) & $4.3\cdot 10^{-3}$ ($3.5\cdot 10^{-3}$)\\
 \hline
 $H = 0.9$ & $2.2\cdot 10^{-2}$ ($4.1\cdot 10^{-2}$) & $1.9\cdot 10^{-2}$ ($1.7\cdot 10^{-2}$)\\
 \hline
\end{tabular}
\medskip
\caption{Mean and StD. (in parentheses) of ${\rm ISE}(\widehat b)$ for $\delta = 1,2$ and $H = 0.6,0.9$ (100 repetitions).}\label{table_MISE_FBM}
\end{center}
\end{table}
\newline
Both Figure \ref{plots_FBM} and the columns of Table \ref{table_MISE_FBM} show that for a fixed value of $\delta$, the mean and standard deviation of ${\rm ISE}(\widehat b)$ are small but degrade when $H$ increases. Moreover, the rows of Table \ref{table_MISE_FBM} show that for a fixed value of $H$, the mean and standard deviation of ${\rm ISE}(\widehat b)$ improve when $\delta$ increases. These numerical findings were expected from Corollary \ref{risk_bound_estimator_b_0_long_time_drifted_fBm}.
%


%
\subsection{Estimation in Model (\ref{linear_SDE_correlated_noises}) via $\widehat b'$}\label{subsection_numerical_experiments_b_0'}
As in the first part of Section \ref{subsection_numerical_experiments_b_0}, $N = 100$ independent Brownian motions $W^1,\dots,W^N$ are simulated along the dissection $\{\frac{\ell T}{n}\textrm{ $;$ }\ell = 0,\dots,n\}$ of $[0,T]$ with $T = 1$ and $n = 150$. Consider the $N$-dimensional process $\mathbf Z = (Z^1,\dots,Z^N)$ such that
\begin{displaymath}
\frac{1}{\sigma}\mathbf Z_t =
\Gamma_{\star}^{\frac{1}{2}}\mathbf W_t
\textrm{ $;$ }
\forall t\in [0,T],
\end{displaymath}
where $\sigma = 0.5$, $\mathbf W = (W^1,\dots,W^N)$, and $\Gamma_{\star} = (\gamma^{|i - k|})_{i,k}$ with $\gamma\in (0,1)$. Clearly, $Z^1,\dots,Z^N$ are $N$ correlated Brownian motions satisfying the condition (\ref{noise_covariance_condition_Gaussian_2}), whose correlation increases with $\gamma$. For any $i\in\{1,\dots,N\}$, the solution $S^i$ of Equation (\ref{linear_SDE_correlated_noises}) is simulated thanks to (\ref{linear_SDE_correlated_noises_transformed}) with $X^i = b_0 + Z^i$, where
\begin{displaymath}
b_0(t) =\frac{t^2}{2} -\frac{\sigma^2t}{2}
\textrm{ $;$ }\forall t\in [0,T].
\end{displaymath}
Then, the drift function $\texttt b_0 = \frac{\sigma^2}{2} + b_0' = {\rm Id}_{[0,T]}$ in Equation (\ref{linear_SDE_correlated_noises}) is estimated by
\begin{displaymath}
\widehat{\tt b} =
\frac{\sigma^2}{2} +\widehat b' =
\frac{\sigma^2}{2} +\sum_{j = 1}^{\widehat m}\mathcal I(\varphi_j,\overline X)\varphi_j,
\end{displaymath}
where $(\varphi_1,\dots,\varphi_N)$ is the $[0,T]$-supported trigonometric basis, and $\widehat m$ is selected in $\{2,\dots,12\}$ via (\ref{model_selection_criterion}). After preliminary simulations, the constant $\mathfrak c_{\rm cal}$ involved in the model selection criterion (\ref{model_selection_criterion}) is chosen as follow: $\mathfrak c_{\rm cal} = 5$.
\\
\\
Figure \ref{plots_BS} shows the true function $\texttt b_0$ (red line) and 5 adaptive estimations obtained thanks to $\widehat{\tt b}$ (dashed black lines) for three values of $\gamma$ ($0$, $0.5$ and $0.75$).
\begin{figure}[h!]
\begin{tabular}{ccc}
 \includegraphics[width=0.3\textwidth]{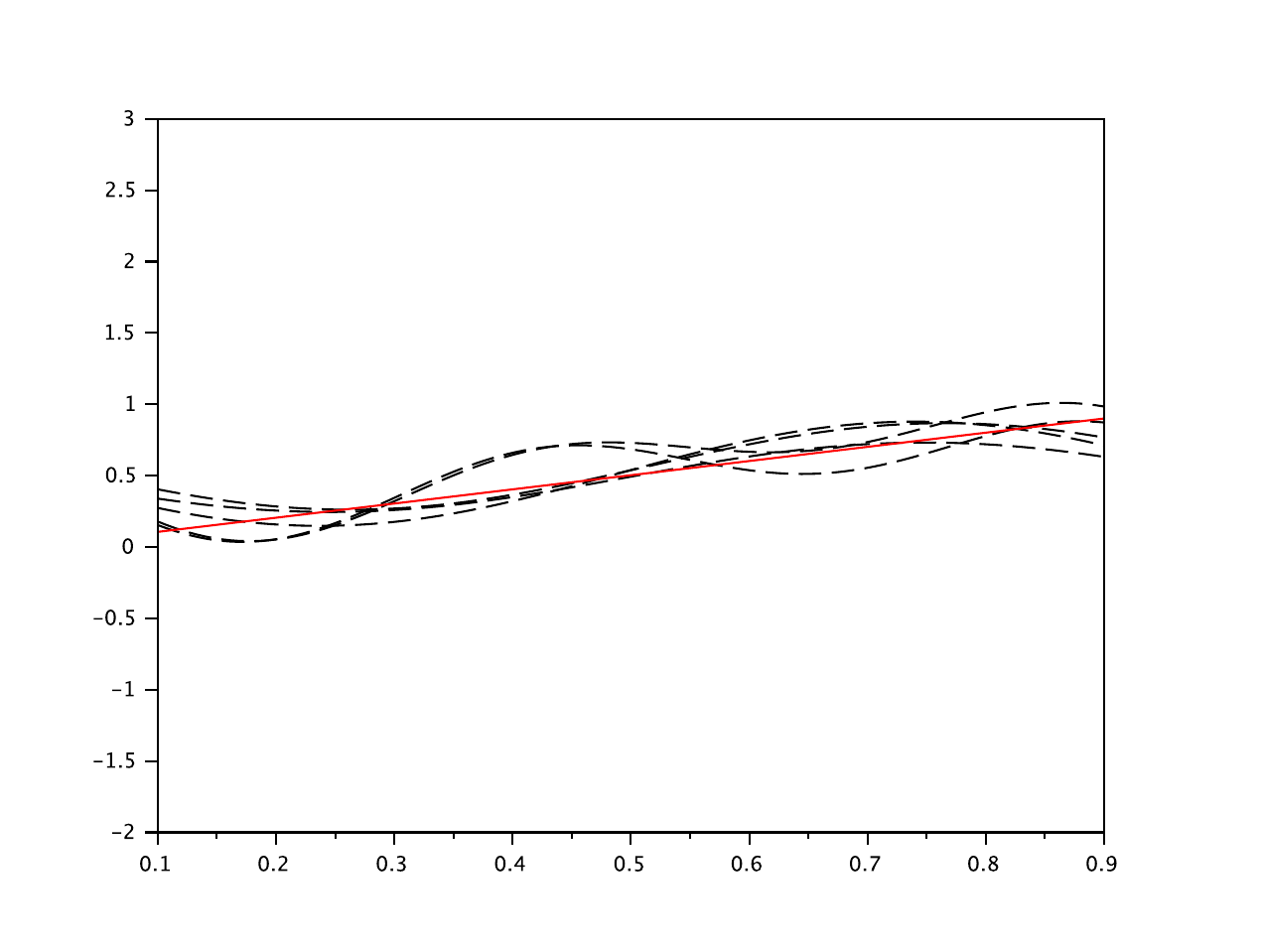} &
 \includegraphics[width=0.3\textwidth]{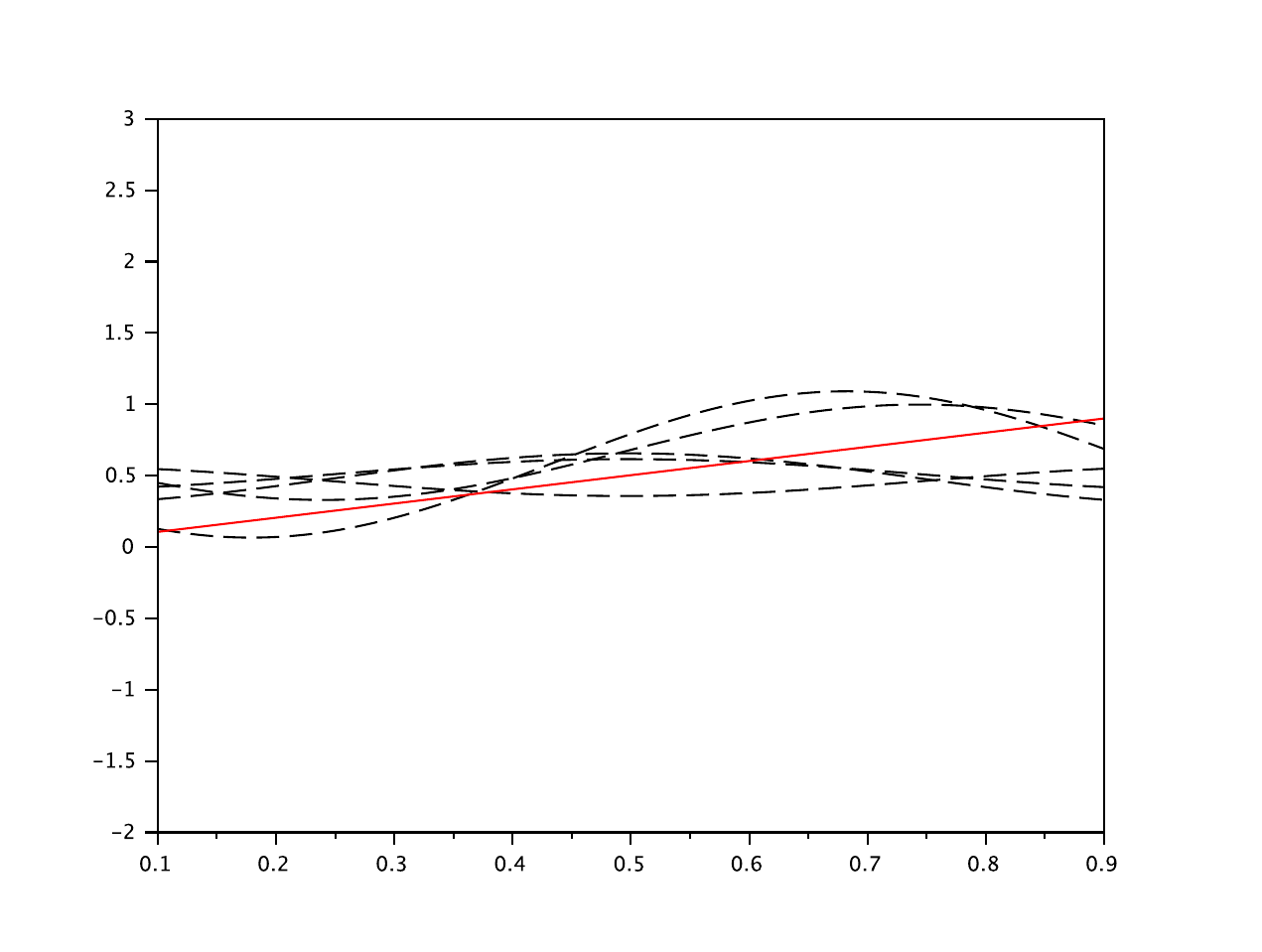} &
 \includegraphics[width=0.3\textwidth]{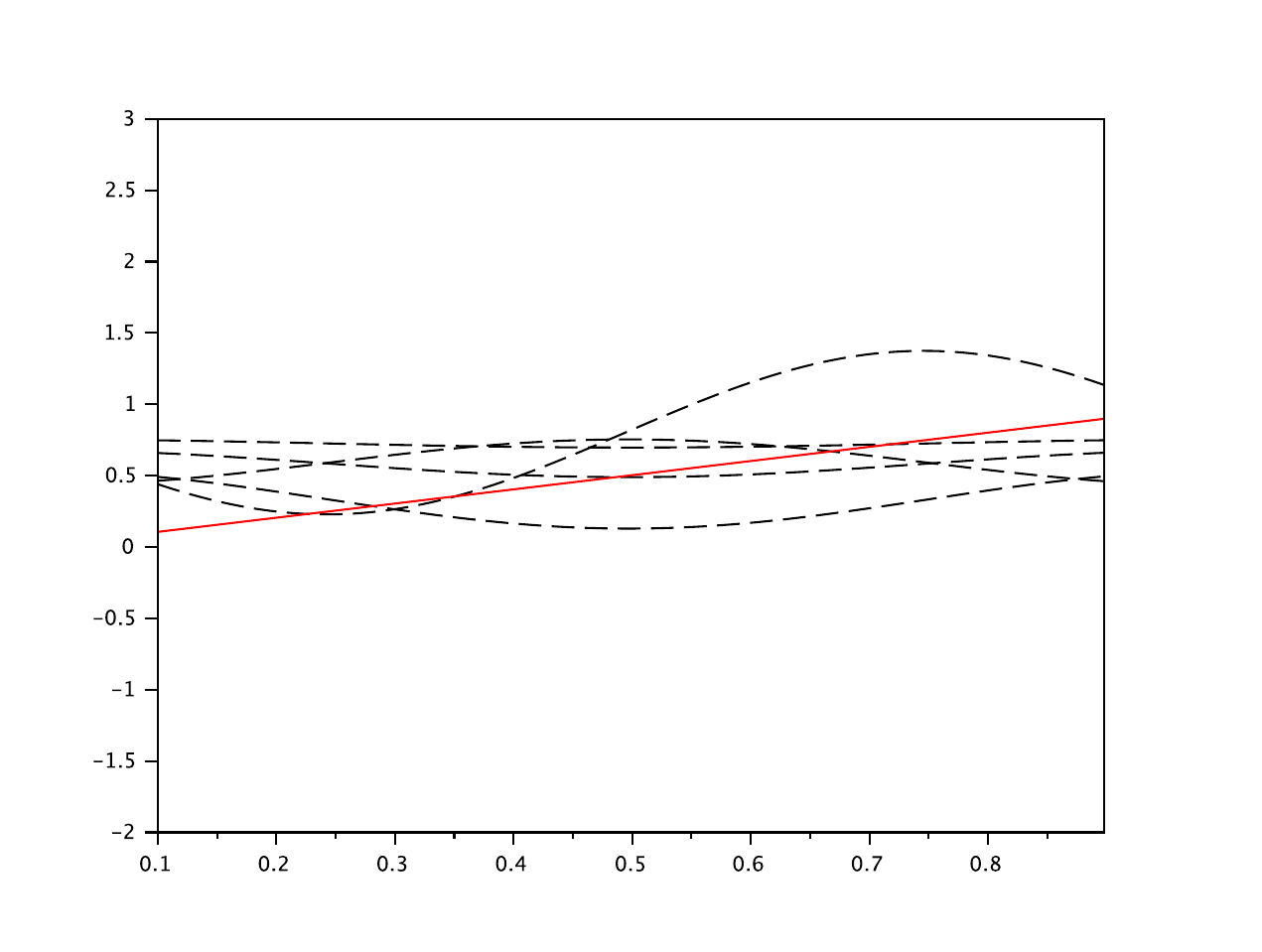}
\end{tabular}
\caption{Plots of 5 adaptive estimations (dashed black) of the true function $\texttt b_0$ (red) in Equation (\ref{linear_SDE_correlated_noises}) for $\gamma = 0,0.5,0.75$ ($N = 100$ copies).}
\label{plots_BS} 
\end{figure}
The experiment is repeated $100$ times, and the means and standard deviations of ${\rm ISE}(\widehat{\tt b})$ and $\widehat m$ are stored in Table \ref{table_MISE_BS}.
\begin{table}[h!]
\begin{center}
\begin{tabular}{|l||c|c|c|}
 \hline
 $\gamma$ & $0$ & $0.5$ & $0.75$\\
 \hline
 \hline
 Mean ${\rm ISE}(\widehat{\tt b})$ & $1.1\cdot 10^{-2}$ & $4.3\cdot 10^{-2}$ & $8.1\cdot 10^{-2}$\\
 \hline
 StD. ${\rm ISE}(\widehat{\tt b})$ & $0.2\cdot 10^{-2}$ & $2.7\cdot 10^{-2}$ & $2.7\cdot 10^{-2}$\\
 \hline
 \hline
 Mean $\widehat m$ & $3.4$ & $2.4$ & $2.2$\\
 \hline
 StD. $\widehat m$ & $0.894$ & $0.548$ & $0.447$\\
 \hline
\end{tabular}
\medskip
\caption{Means and StD. of ${\rm ISE}(\widehat{\tt b})$ and $\widehat m$ (100 repetitions).}\label{table_MISE_BS}
\end{center}
\end{table}
\newline
Both Figure \ref{plots_BS} and Table \ref{table_MISE_BS} show that the mean and standard deviation of ${\rm ISE}(\widehat{\tt b})$ are small (of order $10^{-2}$) when $\gamma = 0$ (i.e. $Z^1,\dots,Z^N$ are independent), but degrade when $\gamma$ increases, which was expected from Theorem \ref{risk_bound_adaptive_estimator_b_0'_Gaussian}. Moreover, for all considered values of $\gamma$, the standard deviation of $\widehat m$ is lower than $1$ (see Table \ref{table_MISE_BS}), meaning that our model selection procedure (\ref{model_selection_criterion}) is quite stable.
%


%

%
\end{document}